\documentclass[10pt]{article}

\usepackage[nottoc,notlot,notlof]{tocbibind}
\usepackage{amsmath,amssymb,amsfonts,amsthm}
\usepackage{latexsym}
\usepackage{tikz}
\usepackage[normalem]{ulem}
\usetikzlibrary{automata,positioning}
\usetikzlibrary{chains,fit,shapes}
\usetikzlibrary{calc}
\usetikzlibrary{arrows}
\usepackage{float}
\usetikzlibrary{positioning,calc}
\usetikzlibrary{graphs}
\usetikzlibrary{graphs.standard}
\usetikzlibrary{arrows,decorations.markings}
\usepackage{multicol}
\usepackage{xcolor}
\newtheorem{theorem}{Theorem}[section]
\newtheorem{corollary}{Corollary}[section]
\newtheorem{lemma}{Lemma}[section]
\newtheorem{prop}{Proposition}[section]
\theoremstyle{definition}
\newtheorem{definition}{Definition}[section]
\newtheorem{observation}{Observation}[section]
\newtheorem{example}{Example}[section]

\numberwithin{equation}{section}
\title{Word-representability of co-bipartite graph}

\author{\hspace{1cm} Biswajit Das \ and Ramesh Hariharasubramanian \\ 
{{\footnotesize d.biswajit@iitg.ac.in},\ {\footnotesize  ramesh\_h@iitg.ac.in}}\\{\footnotesize Department of Mathematics, Indian Institute of Technology Guwahati, Guwahati, Assam 781039, India}}

\begin{document}
	\maketitle
	
	\begin{abstract}
 A graph $G = (V, E)$ is \textit{word-representable}, if there exists a word $w$ over the alphabet $V$ such that for letters $\{x,y\}\in V$, $x$ and $y$ alternate in $w$ if and only if $xy \in E$.  A graph is co-bipartite if its complement is a bipartite graph. Therefore, the vertex set of a co-bipartite graph can be partitioned into two disjoint cliques.
     
        The concept of word-representability for co-bipartite graphs has not yet been fully studied. In the book \textit{Words and Graphs} written by \textit{Sergey Kitaev} and \textit{Vadim Lozin}, examples of co-bipartite graphs that are not word-representable are provided. The authors have stated that it remains an open problem to characterize word-representable co-bipartite graphs. It is known that taking the complement of word-representable graphs does not preserve their word-representability. In this paper, we first identify certain classes of bipartite graphs for which word-representation is preserved after the complement operation. We found that the complement of the path graphs, even cycle graphs and generalized crown graphs are also word-representable. Next, we aim to find word-representable co-bipartite graphs in which the size of one clique partition is fixed while the other one can vary. We studied the word-representability of co-bipartite graphs where the sizes of one clique partition are $2$ and $3$. We found that any co-bipartite graphs where the size of the one clique partition is $2$ are word-representable. Also, when the size of the one clique partition is $3$, we found certain co-bipartite graphs are word-representable. Additionally, for word-representable graphs, it has been established that a graph is word-representable if and only if it can be oriented in a specific manner, known as semi-transitive orientation. We provide the necessary and sufficient conditions for a co-bipartite graph to have a semi-transitive orientation. 
	\end{abstract}

	\maketitle
	\pagestyle{myheadings}
	
\section{Introduction}
 The theory of word-representable graphs is an up-and-coming research area. Sergey Kitaev first introduced the notion of word-representable graphs based on the study of the celebrated Perkins semi-group \cite{kitaev2008word}. The word-representable graphs generalized several key graph families, such as circle graphs, comparability graphs, 3-colourable graphs. However, not all graphs are word-representable, so finding those is an interesting problem.
 
The concept of word-representability has been studied for certain classes of graphs, while for others, it remains uncertain whether they are word-representable. In the paper \cite{kitaev2008representable}, various graph classes are studied, identifying both word-representable and non-word-representable graphs. Specifically, outerplanar graphs and prism graphs are shown as word-representable, whereas wheel graphs with an even number of vertices are shown as non-word-representable. Additionally, the line graphs of wheel graphs and complete graphs are shown as non-word-representable in the paper \cite{kitaev2011representability}. Furthermore, the word-representability of crown graphs and $k$-cube graphs is discussed in the papers \cite{glen2018representation} and \cite{broere24k}, respectively.
In the book \cite{kitaev2015words}, various classes of graphs are discussed, some of which are word-representable and others that are not. In Section 3, Proposition 3.5.4 of this book, the existence of both word-representable and non-word-representable graphs is established within several classes. They also stated the characterization of word-representability of these classes as an open problem. 
Later, line graphs and split graphs are studied among these classes. 
Word-represetability of split graphs is studied in the papers \cite{iamthong2022word},\cite{kitaev2021word} and \cite{chen2022representing}. These papers provide the necessary and sufficient conditions for an orientation of a split graph to be semi-transitive. Additionally, based on these conditions, a polynomial-time algorithm is presented to recognize word-representable split graphs. 
In the graph classes discussed in Proposition 3.5.4 of the book \cite{kitaev2015words}, co-bipartite graphs are also included. Co-bipartite graphs are obtained by taking the complement of a bipartite graph. Therefore, the vertices of these graphs can be partitioned into two cliques, and there can be edges connecting vertices from both cliques. If every vertex in one clique is adjacent to every vertex in the other clique, the resulting graph is a complete graph, which is known to be word-representable.
Additionally, there are some known examples of non-word-representable graphs with $7$ vertices in the book \cite{kitaev2015words}. However, the correct list of non-word-representable graphs on $7$ vertices was provided in the paper \cite{kitaev2024human}. Among the graphs, two examples, $\overline{T_1}$ and $\overline{T_2}$, are co-bipartite graphs, as shown in Figures \ref{nonRepTri1} and \ref{nonRepTri2}, respectively.
Our goal is to characterize word-representable co-bipartite graphs. In this section, we define relevant concepts and the necessary preliminaries regarding word-representable graphs. In Section 2, we identify subclasses of bipartite graphs whose complements preserve word-representability. We show that the complements of path graphs, even cycle graphs, and generalized crown graphs are word-representable.
In Section 3, we fix the size of one of the cliques in the two partitions of the vertices of the co-bipartite graphs. We show that if one clique has 2 vertices, the corresponding graph is word-representable. Additionally, when one clique consists of 3 vertices, we establish that by restricting these graphs based on possible neighbours, we can identify word-representable co-bipartite graphs.
In Section 4, we derive the necessary and sufficient conditions for a co-bipartite graph to have a semi-transitive orientation. Next, in Section 5, we explore the connection between split graphs and co-bipartite graphs. 

The required background and preliminaries related to word-representable graphs are described below.
 
 A \textit{subword} of a word $w$ is a word obtained by removing certain letters from $w$. In a word $w$, if $x$ and $y$ alternate, then $w$ contains $xyxy\cdots$ or $yxyx\cdots$ (odd or even length) as a subword. 

\begin{definition}\textit{(\cite{kitaev2015words} , Definition 3.0.5)}.
 A simple graph $G = (V, E)$ is \textit{word-representable} if there exists a word $w$ over the alphabet $V$ such that letters $x$ and $y$, $\{x,y\}\in V$ alternate in $w$ if and only if $xy \in E$, i.e., $x$ and $y$ are adjacent for each $x\not=y$. If a word $w$ \textit{represents} $G$, then $w$ contains each letter of $V(G)$ at least once.
\end{definition}

 In a word $w$ that represents a graph $G$, if $xy\notin E(G)$, then the non-alternation between $x$ and $y$ occurs if any one of these $xxy$, $yxx$, $xyy$, ${yyx}$ subwords is present in $w$.
\begin{definition}
    (\textit{\cite{kitaev2015words}, Definition 3.2.1.}) \textit{$k$-uniform word} is the word $w$ in which every letter occurs exactly $k$ times.
\end{definition}

\begin{definition}(\textit{\cite{kitaev2015words}, Definition 3.2.3.)} A graph is $k$-word-representable if there exists a $k$-uniform word representing it.
\end{definition}

\begin{definition} (\textit{\cite{kitaev2017comprehensive}, Definition 3})
    For a word-representable graph $G$, \textit{the representation number} is the least $k$ such that $G$ is $k$-representable.
\end{definition}
\begin{prop}(\textit{\cite{kitaev2015words}, Proposition 3.2.7})\label{pr1}
		Let $w = uv$ be a $k$-uniform word representing a graph $G$, where $u$ and $v$ are two, possibly empty, words. Then, the word $w' = vu$ also represents $G$.
\end{prop}
\begin{prop}(\textit{\cite{kitaev2015words}, Proposition 3.0.15.})\label{pr2}
   Let $w = w_1xw_2xw_3$ be a word representing a graph $G$,
where $w_1$, $w_2$ and $w_3$ are possibly empty words, and $w_2$ contains no $x$. Let $X$ be the set of all letters that appear only once in $w_2$. Then, possible candidates for $x$ to be adjacent in $G$ are the letters in $X$. 
\end{prop}

 The \textit{initial permutation} of $w$ is the permutation obtained by removing all but the leftmost occurrence of each letter $x$ in $w$, and it is denoted by $\pi(w)$. Similarly, the \textit{final permutation} of $w$ is the permutation obtained by removing all but the rightmost occurrence of each letter $x$ in $w$, and it is denoted $\sigma(w)$. For a word $w$, $w_{\{x_1, \cdots, x_m\}}$ denotes the word after removing all letters except the letters $x_1, \ldots, x_m$ present in $w$. 
 \begin{example}
     $w = 6345123215$, we have $\pi(w) = 634512$, $\sigma(w) = 643215$ and $w_{\{6,5\}} = 655$.
  \end{example}
	\begin{observation}(\textit{\cite{kitaev2008representable}, Observation 4})\label{pw}
 		Let $w$ be the word-representant of $G$. Then $\pi(w)w$ also represents $G$.
 	\end{observation}

\begin{definition}(\textit{\cite{kitaev2015words}, Definition 3.2.8.}) 
A word $u$ contains a word $v$ as a \textit{factor} if $u = xvy$ where $x$ and $y$ can be empty words.
\end{definition}

\begin{example}
    The word $421231423$ contains the words $123$ and $42$ as factors, while all factors of the word $2131$ are $1$, $2$, $3$, $21$, $13$, $31$, $213$, $131$ and $2131$.
\end{example}

There exists a connection between the orientation of the graphs and word-representability. In the following, the definition of semi-transitive orientation and the connection with word-representability is described.

Semi-transitive orientation is defined based on shortcuts in the papers \cite{halldorsson2011alternation} and \cite{halldorsson2016semi}.
 A semi-cycle is the directed acyclic graph obtained by reversing the
direction of one edge of a directed cycle \cite{kitaev2015words}. An acyclic digraph is a {\em shortcut} if it is induced by the vertices of a semi-cycle and contains a pair of non-adjacent vertices. Thus, any shortcut  

\begin{itemize}

\item is {\em acyclic} (that it, there are {\em no directed cycles});

\item has {\em at least} 4 vertices;

\item has {\em exactly one} source (the vertex with no edges coming in), {\em exactly one} sink (the vertex with no edges coming out), and a {\em directed path} from the source to the sink that goes through {\em every} vertex in the graph;

\item has an edge connecting the source to the sink that we refer to as the {\em shortcutting edge};

\item is {\em not} transitive (that it, there exist vertices $u$, $v$ and $z$ such that $u\rightarrow v$ and $v\rightarrow z$ are edges, but there is {\em no} edge $u\rightarrow z$).

\end{itemize}

\begin{definition}[\cite{halldorsson2016semi}, Definition 1] An orientation of a graph is {\em semi-transitive} if it is {\em acyclic} and 
{\em shortcut-free}.\end{definition}

From definitions, it is easy to see that {\em any} transitive orientation is necessarily semi-transitive. The converse is {\em not} true. Thus, semi-transitive orientations generalize transitive orientations.

\begin{lemma}\label{lem-tran-orie}[\cite{kitaev2021word}, Lemma 4.1] Let $K_m$ be a clique in a graph $G$. Then, any acyclic orientation of $G$ induces a transitive orientation on $K_m$. In particular, any semi-transitive orientation of $G$ induces a transitive orientation on $K_m$. In either case, the orientation induced on $K_m$ contains a single source and a single sink. \end{lemma}

A key result in the theory of word-representable graphs is the following theorem. 

\begin{theorem}[\cite{halldorsson2016semi}, Theorem 3]\label{key-thm} A graph is word-representable if and only if it admits a semi-transitive orientation. \end{theorem}
\begin{corollary}([\cite{broere2018word},  Corollary 5.13])\label{corcart}
    The Cartesian product $K_n\square K_2$ has representation number $n$ for $n\in \{1, 2\}$, and representation number $3$ for all $n > 2$.
\end{corollary}

\begin{definition}([\cite{kitaev2015words}],  Definition 3.4.1.)
    A graph $G$ with the vertex set $V = \{1,\ldots, n\}$ is permutationally
representable if it can be represented by a word of the form $p_1\cdots p_k$, where $p_i$ is a
permutation of $V$ for $1 \leq i \leq k$. If G can be represented permutationally involving
$k$ permutations, we say that $G$ is permutationally $k$-representable.
\end{definition}
\begin{theorem}([\cite{gilmore1964characterization}], Theorem 1.)\label{compchar}
     A graph G is a comparability graph if and only if each odd cycle
has at least one triangular chord.
\end{theorem}
For Theorem \ref{compchar}, a cycle of a graph, $G$ is defined as any
finite sequence of vertices $a_1, a_2,\ldots , a_k$ of $G$ such that all of the edges $(a_i, a_{i+i})$, $1 < i < k-1$, and the edge $(a_k, a_1)$ are in $G$, and for no vertices $a$ and $b$
and integers $i,j < k$, $i\neq j$ , is $a = a_i=a_j$, $b = a_{i+1} = a_{j+1}$ or $a = a_i= a_k$,
$b = a_{i+i} = a_1$. 
\begin{theorem}(\cite{kitaev2015words}, Theorem 3.4.3.)\label{thmcomp}
 A graph is permutationally representable if and only if it is a comparability graph.
\end{theorem}
 \begin{theorem} (\textit{\cite{kitaev2015words}, Theorem 3.4.7.}) \label{tm}
     Let $n$ be the number of vertices in a graph $G$ and $x \in V (G)$
 be a vertex of degree $n-1$ (called a dominant or all-adjacent vertex). Let $H = G\setminus {x}$ be the graph obtained from $G$ by removing $x$ and all edges incident to it. Then $G$ is word-representable if and only if $H$ is permutationally representable.
 \end{theorem}

In the paper, $w=w_1w_2\cdots w_n$ denotes the word $w$ contains $\{w_1,w_2,\ldots, w_n\}$ as factors where $w_i$ is a word possibly empty.
\noindent 
\section{Complement of some known bipartite graph and their word-representation}\label{sc2}

We are denoting a bipartite graph as $B(X,Y)$, where $X$ and $Y$ are two independent set of the graph $B(X,Y)$, and $|X|=m$, $|Y|=n$. We also denote the complement of the bipartite graph, i.e. co-bipartite graph as $\overline{B(K_m,K_n)}$ where $V(K_m)=V(X)$ and $V(K_n)=V(Y)$. The $T_1$ graph is a bipartite graph, but in the paper \cite{chen2016word}, the complement of the $T_1$ graph, $\overline{T_1}$ is shown as a non-word-representable graph.

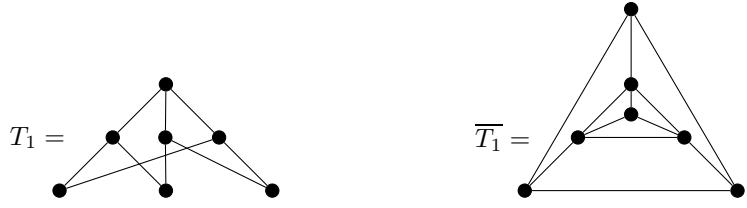
\begin{figure}[h]
\begin{center}
\begin{tabular}{ccccc}
\begin{tikzpicture}[node distance=1cm,auto,main node/.style={fill,circle,draw,inner sep=0pt,minimum size=5pt}]

\node[main node] (1) {};
\node[main node] (2) [below left of=1] {};
\node[main node] (3) [below right of=1] {};
\node[main node] (4) [below left of=2] {};
\node[main node] (5) [below right of=2] {};
\node[main node] (6) [below right of=3] {};
\node[main node] (7) [right of=2, xshift=-.3cm] {};

\node (8) [left of=2] {$T_1=$};

\path
(1) edge (2)
(1) edge (7)
(1) edge (3)
(2) edge (4)
(2) edge (5)
(7) edge (5)
(7) edge (6)
(3) edge (6)
(3) edge (4);

\end{tikzpicture}
& 

\ \ \ \ \ \ \ \ \ \ \ \ \ \ \

&

\begin{tikzpicture}[node distance=1cm,auto,main node/.style={fill,circle,draw,inner sep=0pt,minimum size=5pt}]

\node [main node](1) {};
\node [main node](2) [ below left of=1] {};
\node [main node](3) [ below right of=1] {};
\node [main node](4) [ above of=1] {};
\node [main node](5) [ below left of=2] {};
\node [main node](6) [ below right of=3] {};
\node [main node](7) [below of =1, yshift=6mm]{};
\node (8) [left of=2] {$\overline{T_1}=$};

\draw (1) -- (2);
\draw (2) -- (3);
\draw (1) -- (3);
\draw (4) -- (5);
\draw (5) -- (6);
\draw (4) -- (6);
\draw (1) -- (4);
\draw (2) -- (5);
\draw (3) -- (6);
\draw (1) -- (7);
\draw (2) -- (7);
\draw (3) -- (7);

\end{tikzpicture}
\end{tabular}

\caption{\label{nonRepTri1} The minimal (by the number of vertices) non-word-representable co-bipartite graphs $\overline{T_1}$}
\end{center}
\end{figure}

The $T_2$ graph is a bipartite graph(tree), but in the paper \cite{kitaev2024human}, the complement of $T_2$ graph, $\overline{T_2}$ is shown as a non-word-representable graph.

\begin{figure}[h]
\begin{center}
\begin{tabular}{ccccc}
\begin{tikzpicture}[node distance=1cm,auto,main node/.style={fill,circle,draw,inner sep=0pt,minimum size=5pt}]

\node[main node] (1) {};
\node[main node] (2) [below of=1] {};
\node[main node] (3) [below of=2] {};
\node[main node] (4) [below of=3] {};
\node[main node] (5) [right of=1] {};
\node[main node] (6) [below of=5] {};
\node[main node] (7) [below of=6] {};

\node (8) [left of=2] {$T_2=$};

\path
(1) edge (5)
(2) edge (6)
(3) edge (7)
(4)edge (5)
(4)edge (6)
(4)edge (7);

\end{tikzpicture}
& 

\ \ \ \ \ \ \ \ \ \ \ \ \ \ \

&

\begin{tikzpicture}[node distance=1cm,auto,main node/.style={fill,circle,draw,inner sep=0pt,minimum size=5pt}]

\node [main node](1) {};
\node (2) [ below left of=1] {};
\node [main node](3) [ left of=2] {};
\node (4) [ below right of=1] {};
\node [main node](5) [ right of=4] {};
\node [main node](6) [ below of=1, yshift=-5mm] {};
\node [main node](7) [ below of=6] {};
\node (8) [ left of=7] {};
\node [main node](9) [ below left of=8] {};
\node (10) [ right of=7] {};
\node [main node](11) [ below right of=10] {};
\node (12) [below left of=3] {$\overline{T_2}=$};

\draw (1) -- (3);
\draw (1) -- (5);
\draw (1) -- (6);
\draw (3) -- (5);
\draw (3) -- (6);
\draw (5) -- (6);
\draw (7) -- (9);
\draw (7) -- (11);
\draw (9) -- (11);
\draw (3) -- (7);
\draw (3) -- (9);
\draw (5) -- (7);
\draw (5) -- (11);
\draw (6) -- (9);
\draw (6) -- (11);

\end{tikzpicture}
\end{tabular}

\caption{\label{nonRepTri2} The minimal (by the number of vertices) non-word-representable co-bipartite graphs $\overline{T_2}$}
\end{center}
\end{figure}
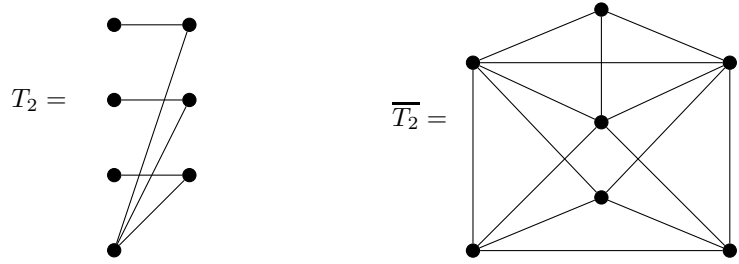
From this, a question arises for all sub-classes of bipartite graphs: Do any sub-classes exist whose complements are word-representable? We have studied some sub-classes of bipartite graphs whose complements are indeed word-representable.

First, for the complete bipartite $K_{m,n}$ graph, the complement of this graph is the union of two disjoint complete graphs of order $m$ and $n$. So, $\overline{K_{m,n}}$ is word-representable.

We found that the complement of some other classes of bipartite graphs is also word-representable.
\subsection{Complement of the path graph}
\begin{definition}
 The path graph $P_n$ is a tree that has two vertices of degree $1$, while the remaining $n-2$  vertices have degree $2$.  
\end{definition}
As $P_{2n}$ is a bipartite graph; we are partitioning the $P_n$ graph into two independent sets, $X$ and $Y$. We denote the vertex of $X$ and $Y$ as $\{1,2,\cdots, n\}$ and $\{1',2',\cdots, n'\}$. According to the definition of a path graph, the possible edges between the sets $X$ and $Y$ are in the following.
\begin{enumerate}
    \item $1\sim 1'$
    \item $i\sim (i-1)'$, $i\sim i'$ $1<i\leq n$
\end{enumerate} 
In Figure \ref{fig1}, we consider an even-length path of length $6$ and use the above notation. In that figure, set $X$ contains the vertices $\{1,2,3\}$ and set $Y$ contains the vertices $\{1',2',3'\}$. 

\begin{figure}[h]
\begin{center}
\begin{tikzpicture}[node distance=1cm,auto,main node/.style={circle,draw,inner sep=1pt,minimum size=5pt}]

\node[main node] (1) {1};
\node[main node] (2) [right of=1] {1'};
\node[main node] (3) [ below of=1] {2};
\node[main node] (4) [below of=2] {2'};
\node[main node] (5) [below of=3] {3};
\node[main node] (6) [below of=4] {3'};

\node (8) [left of=1] {$P_6=$};

\path
(1) edge (2)
(2) edge (3)
(3) edge (4)
(4) edge (5)
(5) edge (6);

\end{tikzpicture}

\caption{\label{fig1} Path graph on $6$ vertices}
\end{center}
\end{figure}
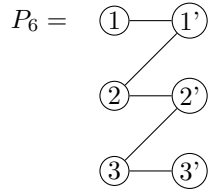
\begin{theorem}\label{thm1}
    The complement of an even-length path graph is word-representable.
\end{theorem}
\begin{proof}
    Suppose $P_{2n}$ is an even-length path graph of order $2n$. As we have mentioned earlier, we can partition $P_{2n}$ is two independent sets $X$ and $Y$ where $V(X)=\{1,2,\cdots,n\}$ and $V(Y)=\{1',2',\cdots,n'\}$. Also, the possible edges between $X$ and $Y$ are  $1\sim 1'$, $i\sim (i-1)'$, $i\sim i'$, $1<i\leq n$. Now, the complement of path graph $P_{2n}$ contain the same vertex set $\{1,2,\cdots,n,1',2',\cdots,n'\}$. But the possible edges of $\overline{P_{2n}}$ are in the following:
    \begin{enumerate}
        \item $i\sim j$, $\forall \{i,j\}\in V(X)$, $i\neq j$.
        \item $i'\sim j'$, $\forall \{i',j'\}\in V(Y)$, $i'\neq j'$.
        \item $1\sim i'$, $1<i\leq n$
        \item $i\sim j'$, $1\leq i\leq n$, $1\leq j\leq n$, $j\neq i-1$ and $j\neq i$.
    \end{enumerate}
    Our claim is that the word $w=121'32'\cdots i(i-1)'\cdots (n-1)(n-2)'n(n-1)'n'$ $1'12'2\cdots i'i\cdots (n-1)'(n-1)n'n$ represents the graph $\overline{P_{2n}}$. To prove this, we need to show that for every edge of $\overline{P_{2n}}$, their corresponding vertices are alternating in $w$ and for every non-edge of $\overline{P_{2n}}$, their corresponding vertices are not alternating in $w$. First, we check for every edge of $\overline{P_{2n}}$ in the following.
    \begin{enumerate}
        \item As, $i\sim j$, $\forall \{i,j\}\in V(X)$, $i\neq j$, $i$ and $j$ should alternate in $w$. Now, $w_{\{1,2,\cdots,n\}}=12\cdots n12\cdots n$. So, $i$ and $j$ are alternating in $w$.
        \item As, $i'\sim j'$, $\forall \{i',j'\}\in V(Y)$, $i'\neq j'$, $i'$ and $j'$ should alternate in $w$. Now, $w_{\{1',2',\cdots,n'\}}=1'2'\cdots n'1'2'\cdots n'$. So, $i'$ and $j'$ are alternating in $w$.
        \item As, $1\sim i'$, $1<i\leq n$, $1$ and $i'$ should alternate in $w$. Now, $w_{\{1,2',\cdots,n'\}}=12'\cdots n'12'\cdots n'$. So, $1$ and $i'$ are alternating in $w$.
        \item As, $i\sim j'$, $1\leq i\leq n$, $1\leq j\leq n$, $j\neq i-1$ and $j\neq i$, $i$ and $j'$ should alternate in $w$. Now, $w_{\{i,1',\cdots (i-2)',(i+1)',\cdots,n\}}=1'\cdots (i-2)'i(i+1)'n'1'\cdots (i-2)'i(i+1)'n'$. So, $i$ and $j'$ are alternating in $w$.
    \end{enumerate}
    Now, we check for every non-edge of $\overline{P_{2n}}$ in the following.
    \begin{enumerate}
        \item As, $1\nsim 1'$, $1$ and $1'$ should not alternate in $w$. Now, $w_{\{1,1'\}}=11'1'1$. So, $1$ and $j'$ are not alternating in $w$. 
        \item As, $i\nsim (i-1)'$ and $i\nsim i'$, $1<i\leq n$, $i$ should not alternate with $(i-1)'$ and $i'$ in $w$. Now, $w_{\{i,(i-1)', i'\}}=i(i-1)'i'(i-1)'i'i$. So, $i$ is not alternating with $(i-1)'$ and $i'$ in $w$.
    \end{enumerate}
\end{proof}

\begin{corollary}
    The complement of an odd-length path graph is word-representable.
\end{corollary}
\begin{proof}
    According to Theorem \ref{thm1}, the complement of an even-length path graph, $\overline{P_{2n}}$ is word-representable. As word-representable graphs are hereditary, removing the vertex $n'$ from the word $w$ mentioned in Theorem \ref{thm1} provides the word-representation for an odd-length path graph $P_{2n-1}$. 
\end{proof}
\begin{corollary}
    $\overline{P_{2n}}$ is comparability graph.
\end{corollary}
\begin{proof}
    From the word $w=121'32'\cdots i(i-1)'\cdots (n-1)(n-2)'n(n-1)'n'$ $1'12'2\cdots i'i\cdots (n-1)'(n-1)n'n$ that represents even-length path graph, we can see that it is concatenation of $2$ permutations, $121'32'\cdots i(i-1)'\cdots (n-1)(n-2)'n(n-1)'n'$ and $1'12'2\cdots i'i\cdots (n-1)'(n-1)n'n$. Also, $w=121'32'\cdots i(i-1)'\cdots (n-1)(n-2)'n(n-1)'$ $1'12'2\cdots i'i\cdots (n-1)'(n-1)n$ represents odd-length path graph, and it is also the concatenation of $2$ permutations, $121'32'\cdots i(i-1)'\cdots (n-1)(n-2)'n(n-1)'$ and $1'12'2\cdots i'i\cdots (n-1)'(n-1)n$. Therefore, $\overline{P_{2n}}$ is permutationally word-representable. So, according to Theorem \ref{thmcomp}, $\overline{P_{2n}}$ is a comparability graph.
\end{proof}
\subsection{Complement of the cycle graph}
\begin{definition}
A graph with $n$ vertices that has a single cycle of length $n$ is called a cycle graph $C_n$.
\end{definition}
As $C_{2n}$ is a bipartite graph, we are partitioning the $C_{2n}$ graph like we have done for the $P_n$ graph into two independent sets, $X$ and $Y$ and denote the vertex of $X$ and $Y$ as $\{1,2,\cdots, n\}$ and $\{1',2',\cdots, n'\}$. According to the definition of cycle graph, the possible edges between the sets $X$ and $Y$ are:
\begin{enumerate}
    \item $1\sim 1'$, $1\sim n'$
    \item  $i\sim (i-1)'$, $i\sim i'$ $1<i\leq n$
\end{enumerate}In the following figure, we consider an even length of length $6$ and use the above notation.

\begin{figure}[h]
\begin{center}
\begin{tikzpicture}[node distance=1cm,auto,main node/.style={circle,draw,inner sep=1pt,minimum size=5pt}]

\node[main node] (1) {1};
\node[main node] (2) [right of=1] {1'};
\node[main node] (3) [ below of=1] {2};
\node[main node] (4) [below of=2] {2'};
\node[main node] (5) [below of=3] {3};
\node[main node] (6) [below of=4] {3'};

\node (8) [left of=1] {$P_6=$};

\path
(1) edge (2)
(1) edge (6)
(2) edge (3)
(3) edge (4)
(4) edge (5)
(5) edge (6);

\end{tikzpicture}

\caption{\label{fig2} Cycle graph on $6$ vertices}
\end{center}
\end{figure}
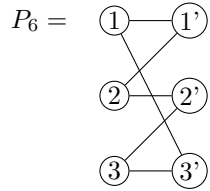

\begin{theorem}
    The complement of the even-length cycle graph is word-representable.
\end{theorem}
\begin{proof}
 Suppose $C_{2n}$ is an even-length cycle graph of order $2n$. As we have mentioned earlier, we can partition $C_{2n}$ is two independent sets $X$ and $Y$ where $V(X)=\{1,2,\cdots,n\}$ and $V(Y)=\{1',2',\cdots,n'\}$. Now, the complement of cycle graph $C_{2n}$ contain the same vertex set $\{1,2,\cdots,n,1',2',\cdots,n'\}$. But the possible edges of $\overline{C_{2n}}$ are in the following:
    \begin{enumerate}
        \item $i\sim j$, $\forall \{i,j\}\in V(X)$, $i\neq j$.
        \item $i'\sim j'$, $\forall \{i',j'\}\in V(Y)$, $i'\neq j'$.
        \item $1\sim i'$, $1<i< n$
        \item $i\sim j'$, $1\leq i\leq n$, $1\leq j\leq n$, $j\neq i-1$ and $j\neq i$.
    \end{enumerate}
After removing the edge between the vertex $1$ and $n'$ from the $\overline{P_{2n}}$ graph, we can obtain $\overline{C_{2n}}$ graph. We are deriving the word-representation of $\overline{C_{2n}}$ graph from the word-representation of $\overline{P_{2n}}$ graph. According to Theorem \ref{thm1}, $w=121'32'\cdots i(i-1)'\cdots (n-1)(n-2)'n(n-1)'n'$ $1'12'2\cdots i'i\cdots (n-1)'(n-1)n'n$ is representing the graph $\overline{P_{2n}}$. According to Observation \ref{pw}, $w_1=\pi(w)w$ also represents $\overline{P_{2n}}$. Now, $w_1=121'32'\cdots i(i-1)'\cdots (n-1)(n-2)'n(n-1)'n'121'32'\cdots i(i-1)'\cdots (n-1)(n-2)'n(n-1)'n' 1'12'2\cdots i'i\cdots (n-1)'(n-1)n'n$. In the word $w_1$, we swap the first occurrence of $n'$ and the second occurrence of $1$ to obtain a new word $w_1'$.  As, $w_1'= 121'32'\cdots i(i-1)'\cdots (n-1)(n-2)'n(n-1)'1n'21'32'\cdots i(i-1)'\cdots (n-1)(n-2)'n(n-1)'n' 1'12'2\cdots i'i\cdots (n-1)'(n-1)n'n$, any alternation and non-alternation between the vertices of $\{2,3,\cdots,n,1',\cdots,(n-1)'\}$ are not changed. Therefore, we only need to check for the vertices $1$ and $n'$. Also, the alternation and non-alternation of $1$ and $n'$ with other vertices are also not changed. But, $w_{1_{\{1,n'\}}}'=11n'n'1n'$, so $1$ and $n'$ are not alternating. Hence, $w_1'$ represents the graph $\overline{C_{2n}}$.
\end{proof}

\subsection{Complement of generalized crown graph}
\begin{definition}
A crown graph $ H_{n,n} $ is a bipartite graph obtained from a complete bipartite graph $K_{n,n}$ by removing the edges of a perfect matching.
\end{definition}
Now, crown graphs $H_{n,n}$ are bipartite graphs, and the complement of the crown graph is the cartesian product of $K_n\times K_2$. According to Corollary \ref{corcart}, $\overline{H_{n,n}}$ is word-representable.
\begin{definition}[\cite{trotter1992combinatorics}]
   Generalized Crown Graph $S_{n-k}^k$ are bipartite graphs on $2n$ vertices that are obtained from a complete bipartite graph $K_{n,n}$ by removing the edges of $k+1$ number of perfect matching in a certain way. We assume the generalized crown graph $S_{n-k}^k$ are partitioned into two independent sets $X$ and $Y$, where $V(X)=\{1,2,\ldots,n\}$ and $V(Y)=\{1',2',\ldots,n'\}$. We can obtain the generalized crown graph $S_{n-k}^k$ by removing the following perfect matching from the complete bipartite graph $K_{n,n}$.
\begin{enumerate}
    \item $\forall 1\leq i\leq (n-k)$, remove the edges from $i$ to $j$, $\forall j\in \{i',(i+1)',\ldots, (i+k)'\}$.
    \item $\forall (n-k+1)\leq i\leq n$, remove the edges from $i$ to $j$, $\forall j \in \{i',((i+1)\text{mod n})',((i+2)\text{mod n})',\ldots,((i+k)\text{mod n})'\}$, where $0'=n'$. 
\end{enumerate}
\end{definition}
 From the definition, $H_{n,n}=S_n^0$.
Now, we want to find whether the complement of generalized crown graphs is word-representable or not.

The complement of a generalized crown graph $\overline{S_{n-k}^k}$ also can partition into two cliques $X$ and $Y$,  where $V(X)=\{1,2,\ldots,n\}$ and $V(Y)=\{1',2',\ldots,n'\}$. So, the graph $\overline{S_{n-k}^k}$ contains the following edges. 
\begin{enumerate}
    \item $1\leq i<j\leq n$, $i\sim j$.
    \item $1\leq i<j\leq n$, $i'\sim j'$.
    \item $\forall 1\leq i\leq (n-k)$, $i\sim j$, $\forall j\in \{i',(i+1)',\ldots, (i+k)'\}$.
    \item $\forall (n-k+1)\leq i\leq n$, $i\sim j$, $\forall j \in \{i',((i+1)\text{mod n})',((i+2)\text{mod n})',\ldots,((i+k)\text{mod n})'\}$, where $0'=n'$. 
\end{enumerate}
\begin{theorem}
    The complement of a generalized crown graph $\overline{S_{n-k}^k}$ is word-representable.
\end{theorem}

\begin{proof}
    We are applying the following homomorphism, and we will generate a word that represents the graph $\overline{S_{n-k}^k}$. The required homomorphism are defined as$\forall i$, $1\leq i\leq 5$, $h_i: V(\overline{S_{n-k}^k})\rightarrow V(\overline{S_{n-k}^k})$:
    \begin{itemize}
        \item $h_1(x)=x$
        \item $h_2(x)=x'(n-k+x)$
        \item $h_3(x)=(k+x)'x$
        \item $h_4(x)=x'$
        \item $h_5(x)=xx'$
    \end{itemize}
    Now we assume that $V_1=12\cdots(n-k)$, $V_2=12\cdots k$, $V_3=(k+1)(k+2)\cdots n$ and $V_4=V_2V_3=12\cdots k(k+1)\cdots n$.
    We claim that the word $w=h_1(V_1)h_2(V_2)h_3(V_1)h_2(V_2)h_4(V_3)h_5(V_4)$ represents the graph $\overline{S_{n-k}^k}$.
    First we proof that if $i\sim j$, in the graph $\overline{S_{n-k}^k}$ then $i$ and $j$ are alternating in $w$. We are considering the possible edges of the graph $\overline{S_{n-k}^k}$ as mentioned earlier in the following.
    \begin{itemize}
        \item In the graph $\overline{S_{n-k}^k}$, $i\sim j$, $1\leq i<j\leq n$. We assume $X=\{1,2,\ldots, n\}$, so in the word $w$, 
        \[
           \begin{split}
           &w_{\{X\}}=(h_1(V_1))_{\{X\}}\, (h_2(V_2))_{\{X\}} \, (h_3(V_1))_{\{X\}} \, (h_2(V_2))_{\{X\}} \, (h_4(V_3))_{\{X\}} \, (h_5(V_4))_{\{X\}}
           \\ &(h_1(V_1))_{\{X\}}= 12\cdots (n-k) 
           \\ &(h_2(V_2))_{\{X\}}=(n-k+1)(n-k+2)\cdots n
           \\ &(h_3(V_1))_{\{X\}}= 12\cdots (n-k)
           \\ &(h_2(V_2))_{\{X\}}=(n-k+1)(n-k+2)\cdots n
           \\ & h_4(V_3))_{\{X\}}=\epsilon
            \\ &(h_5(V_4))_{\{X\}}=12\cdots n
            \\ & \begin{split}
                 w_{\{X\}}&=12\cdots (n-k) \; (n-k+1)(n-k+2)\cdots n \; 12\cdots (n-k) \; (n-k+1)(n-k+2)\cdots n \\ & \hspace{5mm} \; 12\cdots (n-k)(n-k+1)(n-k+2)\cdots n
                 \\ &=(12\cdots n)^3
            \end{split}           
        \end{split}
        \] Therefore, $i$ and $j$ are alternating in $w$, $1\leq i<j\leq n$. 
        
        \item In the graph $\overline{S_{n-k}^k}$, $i'\sim j'$, $1\leq i<j\leq n$. We assume $Y=\{1',2',\ldots, n'\}$, so in the word $w$, 
        \[
           \begin{split}
           &w_{\{Y\}}=(h_1(V_1))_{\{Y\}}\, (h_2(V_2))_{\{Y\}} \, (h_3(V_1))_{\{Y\}} \, (h_2(V_2))_{\{Y\}} \, (h_4(V_3))_{\{Y\}} \, (h_5(V_4))_{\{Y\}}
           \\ &(h_1(V_1))_{\{Y\}}= \epsilon 
           \\ &(h_2(V_2))_{\{Y\}}=1'2'\cdots k'
           \\ &(h_3(V_1))_{\{Y\}}= (k+1)'(k+2)'\cdots n'
           \\ &(h_2(V_2))_{\{Y\}}=1'2'\cdots k'
           \\ & h_4(V_3))_{\{Y\}}=(k+1)'(k+2)'\cdots n'
            \\ &(h_5(V_4))_{\{Y\}}=1'2'\cdots n'
            \\ & \begin{split}
                 w_{\{X\}}&=1'2'\cdots k' \; (k+1)'(k+2)'\cdots n' \; 1'2'\cdots k' \; (k+1)'(k+2)'\cdots n'\; 1'2'\cdots n' \\ &=(1'2'\cdots n')^3
            \end{split}           
        \end{split}
        \] Therefore, $i'$ and $j'$ are alternating in $w$, $1\leq i<j\leq n$. 
        
    \item In the graph $\overline{S_{n-k}^k}$, $\forall 1\leq i\leq (n-k)$, $i\sim j$, $\forall j\in \{i',(i+1)',\ldots, (i+k)'\}$. We assume $A=\{i,i',(i+1)',\ldots, (i+k)'\}$, so in the word $w$,
            \[
           \begin{split}
           &w_{\{A\}}=(h_1(V_1))_{\{A\}}\, (h_2(V_2))_{\{A\}} \, (h_3(V_1))_{\{A\}} \, (h_2(V_2))_{\{A\}} \, (h_4(V_3))_{\{A\}} \, (h_5(V_4))_{\{A\}}
           \\ &(h_1(V_1))_{\{A\}}= i 
           \\ &(h_2(V_2))_{\{A\}}=i'(i+1)'\cdots k'
           \\ &(h_3(V_1))_{\{A\}}= (k+1)'(k+2)'\cdots (k+i)'i
           \\ &(h_2(V_2))_{\{A\}}=i'(i+1)'\cdots k'
           \\ & h_4(V_3))_{\{A\}}=(k+1)'(k+2)'\cdots (k+i)'
            \\ &(h_5(V_4))_{\{A\}}=ii'(i+1)'\cdots k'(k+1)'(k+2)'\cdots (k+i)'
            \\ & \begin{split}
                 w_{\{A\}}&=i\; i'(i+1)'\cdots k' \; (k+1)'(k+2)'\cdots (k+i)'i \; i'(i+1)'\cdots k' \; (k+1)'(k+2)'\cdots (k+i)' \; \\ & \hspace{5mm} ii'(i+1)'\cdots k'(k+1)'(k+2)'\cdots (k+i)' \\ &=(ii'(i+1)'\cdots k' \; (k+1)'(k+2)'\cdots (k+i)')^3
            \end{split}           
        \end{split}
        \] 
        Therefore, $i$ and $j$ are alternating in $w$, $1\leq i<j\leq n$. 
        
      \item In the graph $\overline{S_{n-k}^k}$, $\forall (n-k+1)\leq i\leq n$, $i\sim j$, $\forall j \in \{i',((i+1)\text{mod n})',((i+2)\text{mod n})',\ldots,((i+k)\text{mod n})'\}$, where $0'=n'$. Let, $i=n-k+l$, $1\leq l\leq k$. So, $n-k+l$ is also adjacent to the set $B=\{(n-k+l)',(n-k+2)',\ldots,n',1',2',\ldots, a'\}$. We know that the degree of $n-k+l$ is $n-1+k+1=n+k$. As, $n-k+l$ is adjacent to $\{1,2,\cdots,(n-k+l-1),(n-k+l+1),\ldots, n\}$, so the cardinality of the set $|B|=k+1$.
            \[
            \begin{split}
                &\implies|(n-k+l)',(n-k+2)',\ldots,n',1',2',\ldots, a'|= k+1
              \\&\implies|(n-k+l)',(n-k+2)',\ldots,n'|+|1',2',\ldots, a'|=k+1
              \\& \implies n-(n-k+l)+1+a=k+1
              \\& \implies a=l
            \end{split}                
            \]
      Therefore, $B=\{(n-k+l)',(n-k+2)',\ldots,n',1',2',\ldots, l'\}$, so in the word $w$,
            \[
           \begin{split}
           &w_{\{B\}}=(h_1(V_1))_{\{B\}}\, (h_2(V_2))_{\{B\}} \, (h_3(V_1))_{\{B\}} \, (h_2(V_2))_{\{B\}} \, (h_4(V_3))_{\{B\}} \, (h_5(V_4))_{\{B\}}
           \\ &(h_1(V_1))_{\{B\}}= \epsilon 
           \\ &(h_2(V_2))_{\{B\}}=1'2'\cdots l'(n-k+l)
           \\ &(h_3(V_1))_{\{B\}}= (n-k+l)'(n-k+l+1)'\cdots n'
           \\ &(h_2(V_2))_{\{B\}}=1'2'\cdots l'(n-k+l)
           \\ & h_4(V_3))_{\{B\}}=(n-k+l)'(n-k+l+1)'\cdots n'
            \\ &(h_5(V_4))_{\{B\}}=1'2'\cdots l'(n-k+l)(n-k+l)'(n-k+l+1)'\cdots n'
            \\ & \begin{split}
                 w_{\{B\}}&=1'2'\cdots l'(n-k+l)\; (n-k+l)'(n-k+l+1)'\cdots n' \; 1'2'\cdots l'(n-k+l) \; 1'2'\cdots l'(n-k+l) \\ & \hspace{5mm}\; (n-k+l)'(n-k+l+1)'\cdots n' \;  1'2'\cdots l'(n-k+l)(n-k+l)'(n-k+l+1)'\cdots n' \\ &=(1'2'\cdots l'(n-k+l)(n-k+l)'(n-k+l+1)'\cdots n')^3
            \end{split}           
        \end{split}
        \] Therefore, $i$ and $j$ are alternating in $w$, $1\leq i<j\leq n$.    
    \end{itemize}
    As, $w$ preserve the alternation for the edges of the graph $\overline{S_{n-k}^k}$, now
 we need to prove that if $i\nsim j$, in the graph $\overline{S_{n-k}^k}$ then $i$ and $j$ are not alternating in $w$. We are considering the possible non-edges of the graph $\overline{S_{n-k}^k}$ in the following.
 \begin{itemize}
     \item In the graph $\overline{S_{n-k}^k}$, $\forall 1\leq i\leq (n-k)$, $i\nsim j$, $\forall j\in \{1',2',\ldots, (i-1)',(i+k+1)',\ldots, n'\}$. 
      We assume $A=\{i,1',2',\ldots, (i-1)',(i+k+1)',\ldots, n'\}$, so in the word $w$,
            \[
           \begin{split}
           &w_{\{A\}}=(h_1(V_1))_{\{A\}}\, (h_2(V_2))_{\{A\}} \, (h_3(V_1))_{\{A\}} \, (h_2(V_2))_{\{A\}} \, (h_4(V_3))_{\{A\}} \, (h_5(V_4))_{\{A\}}
           \\ &(h_1(V_1))_{\{A\}}= i 
           \\ &(h_2(V_2))_{\{A\}}=\begin{cases} 
                                1'2'\cdots (i-1)' & 1\leq i\leq k \\
                                1'2'\cdots k' & k< i\leq n-k 
                                \end{cases}
           \\ &(h_3(V_1))_{\{A\}}= \begin{cases} 
                                i(i+k+1)(i+k+2)\cdots n' & 1\leq i\leq k \\
                               (k+1)'(k+2)'\cdots (i-1)'i(i+k+1)(i+k+2)\cdots n' & k< i\leq n-k 
                                \end{cases}
           \\ &(h_2(V_2))_{\{A\}}=\begin{cases} 
                                1'2'\cdots (i-1)' & 1\leq i\leq k \\
                                1'2'\cdots k' & k< i\leq n-k 
                                \end{cases}
           \\ & h_4(V_3))_{\{A\}}=\begin{cases} 
                                (i+k+1)'(i+k+2)'\cdots n' & 1\leq i\leq k \\
                               (k+1)'(k+2)'\cdots (i-1)'(i+k+1)'(i+k+2)'\cdots n' & k< i\leq n-k 
                                \end{cases}
            \\ &(h_5(V_4))_{\{A\}}=1'2'\cdots (i-1)'i(i+k+1)'\cdots n'
             \end{split}
        \]
When we combine all the cases of $(h_2(V_2))_{\{A\}}$ and $(h_3(V_1))_{\{A\}}$, we obtain that $\Big(h_2(V_2)h_3(V_1)\Big)_{\{A\}}=1'2'\cdots (i-1)'i(i+k+1)\cdots n'$. Similarly, combining the $(h_2(V_2))_{\{A\}}$ and $(h_4(V_3))_{\{A\}}$, we obtain that $\Big(h_2(V_2)h_4(V_3)\Big)_{\{A\}}=1'2'\cdots (i-1)'(i+k+1)'\cdots n'$. Therefore, the word $w_{\{A\}}=i\;1'2'\cdots (i-1)'i(i+k+1)\cdots n'\; 1'2'\cdots (i-1)'(i+k+1)'\cdots n'1'2'\cdots (i-1)'i(i+k+1)'\cdots n'(ii'(i+1)'\cdots k'$. 

As, $w_{\{i,1',2',\ldots,(i-1)'\}}=i1'2'\cdots (i-1)'i1'2'\cdots (i-1)'1'2'\cdots (i-1)'i$, so in the word $w$, $i$ is not alternating with any element of the set $\{1'2'\cdots (i-1)'\}$. Also, $w_{\{i,(i+k+1)',(i+k+2)',\ldots,n'\}}=ii (i+k+1)'(i+k+2)'\cdots n'  (i+k+1)'(i+k+2)'\cdots n'i (i+k+1)'(i+k+2)'\cdots n'$. Therefore, in the word $w$, $i$ is not alternating with any element of the set $\{(i+k+1)'(i+k+2)'\cdots n' \}$.    
\item Now we consider the $i$ where $\forall (n-k+1)\leq i\leq n$ and $i\in V(\overline{S_{n-k}^k})$. We assume $i=n-k+l$, $1\leq l\leq k$. So, we know that $n-k+l$ is adjacent to $\{(n-k+l)',(n-k+2)',\ldots,n',1',2',\ldots, l'\}$. Therefore, the nonadjacent vertices of $n-k+l$ are $B=\{(l+1)', (l+2)',\ldots, (n-k+l-1)'\}$. So, in the word $w$,
            \[
           \begin{split}
           &w_{\{B\}}=(h_1(V_1))_{\{B\}}\, (h_2(V_2))_{\{B\}} \, (h_3(V_1))_{\{B\}} \, (h_2(V_2))_{\{B\}} \, (h_4(V_3))_{\{B\}} \, (h_5(V_4))_{\{B\}}
           \\ &(h_1(V_1))_{\{B\}}= \epsilon 
           \\ &(h_2(V_2))_{\{B\}}=\begin{cases} 
                                (n-k+l)(l+1)'(l+2)'\cdots (n-k+l-1)' & (n-k)\leq i\leq k \\
                                (n-k+l)(l+1)'(l+2)'\cdots k & k< i\leq n 
                                \end{cases}
           \\ &(h_3(V_1))_{\{B\}}= \begin{cases} 
                                \epsilon & (n-k)\leq i\leq k \\
                               (k+1)'(k+2)'\cdots (n-k+l-1)' & k< i\leq n  
                                \end{cases}
           \\ &(h_2(V_2))_{\{B\}}=\begin{cases} 
                                (n-k+l)(l+1)'(l+2)'\cdots (n-k+l-1)' & (n-k)\leq i\leq k \\
                                (n-k+l)(l+1)'(l+2)'\cdots k & k< i\leq n 
                                \end{cases}
           \\ & h_4(V_3))_{\{B\}}=\begin{cases} 
                                \epsilon & (n-k)\leq i\leq k \\
                               (k+1)'(k+2)'\cdots (n-k+l-1)' & k< i\leq n  
                                \end{cases}
            \\ &(h_5(V_4))_{\{B\}}=(l+1)'(l+2)'\cdots (n-k+l-1)'(n-k+l).
             \end{split}
             \]
   When we combine all the cases of $(h_2(V_2))_{\{B\}}$ and $(h_3(V_1))_{\{B\}}$, we obtain that $\Big(h_2(V_2)h_3(V_1)\Big)_{\{B\}}=(n-k+l)(l+1)'(l+2)'\cdots (n-k+l-1)'=i(l+1)'(l+2)'\cdots (n-k+l-1)'$. Similarly, combining the $(h_2(V_2))_{\{B\}}$ and $(h_4(V_3))_{\{B\}}$, we obtain that $\Big(h_2(V_2)h_4(V_3)\Big)_{\{B\}}=(n-k+l)(l+1)'(l+2)'\cdots (n-k+l-1)'i(l+1)'(l+2)'\cdots (n-k+l-1)'$. 
   
   Therefore, the word $w_{\{B\}}=(n-k+l)(l+1)'(l+2)'\cdots (n-k+l-1)'\;(n-k+l)(l+1)'(l+2)'\cdots (n-k+l-1)'\; (l+1)'(l+2)'\cdots (n-k+l-1)'(n-k+l)=i(l+1)'(l+2)'\cdots (n-k+l-1)'\;i\Big((l+1)'(l+2)'\cdots (n-k+l-1)'\Big)^2i$. So, in the word $w$, $i$ is not alternating with the element of the set $B$.
 \end{itemize}
 Hence, the word $w$ represents the graph $\overline{S_{n-k}^k}$.
\end{proof}
\subsection{ A non-word representable co-bipartite graph $\overline{G_1}$}
In this section, we obtain a co-bipartite graph from the complement of the crown graph that is not a word-representable graph.
The graph $G_1$ is the graph obtained from the crown graph $H_{n,n}$ by adding an isolated vertex $v$. We can easily see that $G_1$ is also a bipartite graph because we can add the vertex $v$ in any partition.
First, we prove that $\overline{H_{n,n}}$ is not a comparability graph.
\begin{lemma}\label{lm21}
    The co-bipartite graph $\overline{H_{n,n}}$ is not a comparability graph.
\end{lemma}

\begin{proof}
    Suppose the co-bipartite graph $\overline{H_{n,n}}$ is partitioned into two cliques $K_n$ and $K'_n$. Now, we assume the vertices of each clique $K_n$ and $K'_n$ are $\{1,2,\ldots,n\}$ and $\{1',2',\ldots, n'\}$, respectively. Also, each $i,j$, $1\leq i<j\leq n$, $i\sim j$, $i\sim i'$ and $i'\sim j'$ in the graph $\overline{H_{n,n}}$. Now we are considering the induced subgraph of the vertices, $1$, $2$, $3$, $1'$, $2'$ and $3'$ from the graph $\overline{H_{n,n}}$ in the following Figure \ref{fig41}.

\begin{figure}[H]
\begin{center}
\begin{tikzpicture}[node distance=1.5cm,auto,main node/.style={fill,circle,draw,inner sep=1pt,minimum size=5pt}]

\node [main node](1) {};
\node [below of=1,yshift=12mm] {$1$};
\node [main node](2) [ below left of=1] {};
\node [right of=2,xshift=-10mm,yshift=2mm] {$2$};
\node [main node](3) [ below right of=1] {};
\node [right of=2,xshift=2mm,yshift=2mm] {$3$};
\node [main node](4) [ above of=1] {};
\node [above of=4,yshift=-12mm] {$1'$};
\node [main node](5) [ below left of=2] {};
\node [left of=5,xshift=12mm] {$2'$};
\node [main node](6) [ below right of=3] {};
\node [right of=6,xshift=-12mm] {$3'$};

\draw (1) -- (2);
\draw (2) -- (3);
\draw (1) -- (3);
\draw (4) -- (5);
\draw (5) -- (6);
\draw (4) -- (6);
\draw (1) -- (4);
\draw (2) -- (5);
\draw (3) -- (6);

\end{tikzpicture}

\caption{\label{fig41} induced subgraph of the vertices, $1$, $2$, $3$, $1'$, $2'$ and $3'$ from the graph $\overline{H_{n,n}}$}
\end{center}
\end{figure}
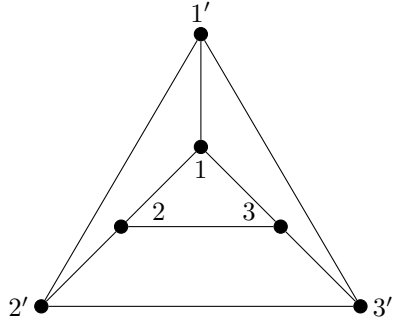
Now, from this graph, we consider a cycle $1'122'233'31$ as described for Theorem \ref{compchar}. This is an odd cycle of length $9$. We can easily verify that no triangular chord exists in this cycle. So, $\overline{H_{n,n}}$ contains an odd cycle that does not have any triangular chord. Therefore, according to Theorem \ref{compchar}, $\overline{H_{n,n}}$ is not a comparability graph.
\end{proof}

\begin{theorem}
The co-bipartite graph $\overline{G_1}$ is non-word-representable where $n\geq 3$, for $H_{n,n}$ graph.
\end{theorem}
\begin{proof}
    As the graph $G_1$ has an isolated vertex $v$, in the co-bipartite graph $\overline{G_1}$, $v$ becomes adjacent to all of the vertices of the graph $\overline{H_{n,n}}$. According to Lemma \ref{lm21},  the graph $\overline{H_{n,n}}$ is not a comparability graph. Therefore, according to Theorem \ref{tm}, $G_1$ is not word-representable.
\end{proof}
\section{Co-bipartite graph on a fixed clique size}
A co-bipartite graph $\overline{B(K_m, K_n)}$ can be partitioned into two cliques $K_m$ and $K_n$, and there are some edges between the vertices of $K_m$ and $K_n$. In these two partitions, we want to fix the size of one partition while allowing the other to be arbitrary. Our aim is to identify which of these graphs are word-representable.
\subsection{Clique size $2$}
We are considering $K_n=K_2$, and want to find for different $K_m$, does $\overline{B(K_m, K_n)}$ word-representable. We assume the vertices of $K_2$ are $1$ and $2$. Also, $N_i$ denotes the set of the vertices that are adjacent to $i$. Now, the vertices of $K_m$ are divided in the following:
\begin{enumerate}
    \item $N_{1\overline{2}}$ are the set of the vertices $v$ such that $v\in V(K_m)$ and $v\in N_1\setminus N_2$. 
    \item $N_{12}$ are the set of the vertices $v$ such that $v\in V(K_m)$ and $v\in N_1\cap N_2$. 
    \item $N_{\overline{1}2}$ are the set of the vertices $v$ such that $v\in V(K_m)$ and $v\in N_2\setminus N_1$. 
    \item $N_{\overline{1}\overline{2}}$ are the set of the vertices $v$ such that $v\in V(K_m)$ and $v\in V\setminus(N_1\cup N_2)$. 
\end{enumerate}
In the following diagram, it will be more precise.

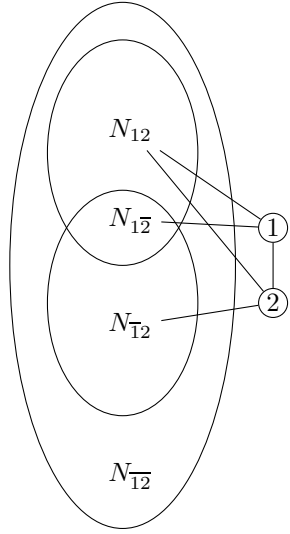
\begin{figure}[h]
\begin{center}
\begin{tikzpicture}[node distance=1cm,auto,main node/.style={circle,draw,inner sep=1pt,minimum size=5pt}]

\node[main node] (1) {1};
\node[main node] (2) [ below of=1] {2};
\node (3) [left of=1, yshift=1mm,xshift=-9mm] {$N_{1\overline{2}}$};
\node (4) [left of=1, yshift=13mm,xshift=-9mm] {$N_{12}$};
\node (5) [left of=1, yshift=-13mm,xshift=-9mm] {$N_{\overline{1}2}$};
\node (6) [left of=1, yshift=-33mm,xshift=-9mm] {$N_{\overline{1}\overline{2}}$};

\draw (1) -- (2);
\draw (1) -- (3);
\draw (1) -- (4);
\draw (2) -- (4);
\draw (2) -- (5);

\draw [] (-2,-1) ellipse (10mm and 15mm); 
\draw [] (-2,1) ellipse (10mm and 15mm); 
\draw [] (-2,-.5) ellipse (15mm and 35mm); 
\end{tikzpicture}

\caption{\label{fig51} the co-bipartite graph $\overline{B(K_m,K_2)}$}
\end{center}
\end{figure}

\begin{theorem}
    The co-biparitite graph $\overline{B(K_m, K_2)}$, is word-representable.
\end{theorem}
\begin{proof}  
Our claim is that the word $w = N_{12}1N_{\overline{1}2}2N_{\overline{1}\overline{2}}N_{1\overline{2}}N_{12}N_{\overline{1}2}N_{\overline{1}\overline{2}}1N_{1\overline{2}}2$ represents the graph $\overline{B(K_m, K_2)}$, where $N_{1\overline{2}}$, $N_{12}$, $N_{\overline{1}2}$, and $N_{\overline{1}\overline{2}}$ are defined above. The possible edges in $\overline{B(K_m, K_2)}$ are in the following:
\begin{enumerate}
    \item $1\sim 2$
    \item $i\sim j$, $\forall {i,j}\in V(K_m), i\neq j$.
    \item $1\sim N_{12}$, $1\sim N_{1\overline{2}}$.
    \item $2\sim N_{12}$, $2\sim N_{\overline{1}2}$.
\end{enumerate}
Now, we prove that the alternation between the vertices $x$ and $y$ occurs in $w$ if the edge $x \sim y$ is included among the edges described above.
\begin{enumerate}
    \item As $1\sim 2$, $1$ and $2$ should alternate in $w$. Now, $w_{\{1,2\}}=1212$. Therefore, $1$ and $2$ are alternating in $w$.
    \item As $i\sim j$, $\forall {i,j}\in V(K_m), i\neq j$, any $i$ and $j$ should alternate in $w$. We know that $N_{1\overline{2}} \cup N_{12}\cup N_{\overline{1}2}\cup N_{\overline{1}\overline{2}}= V(K_m)$. So, $w_{\{V(K_m)\}}=N_{12}N_{\overline{1}2}N_{\overline{1}\overline{2}}N_{1\overline{2}}N_{12}N_{\overline{1}2}N_{\overline{1}\overline{2}}N_{1\overline{2}}$. Hence, $i$ and $j$ are alternating in $w$.
    \item As, $1\sim N_{12}$, $1\sim N_{1\overline{2}}$, $1$ should alternate with $N_{12}$ and $N_{1\overline{2}}$ in $w$. Now, $w_{\{1,N_{12},N_{1\overline{2}}\}}=N_{12}1N_{1\overline{2}}N_{12}1N_{1\overline{2}}$. So,  $1$ is alternating with $N_{12}$ and $N_{1\overline{2}}$ in $w$.
    \item As, $2\sim N_{12}$, $2\sim N_{\overline{1}2}$, $2$ should alternate with $N_{12}$ and $N_{\overline{1}2}$ in $w$. Now, $w_{\{2,N_{12},N_{\overline{1}2}\}}=N_{12}N_{\overline{1}2}2N_{12}1N_{\overline{1}2}2$. So,  $2$ is alternating with $N_{12}$ and $N_{\overline{1},2}$ in $w$.
\end{enumerate}
  Since alternation is preserved in $w$ for the edges of the graph $\overline{B(K_m, K_2)}$, we now need to show that non-alteration is also preserved in $w$ for the non-edges of the graph $\overline{B(K_m, K_2)}$. The possible cases are the following.
    \begin{enumerate}
        \item In the graph $\overline{B(K_m, K_2)}$, $1\nsim N_{\overline{1}2}$ and $1\nsim N_{{\overline{1}\overline{2}}}$. In the word $w$, $w_{\{1, N_{\overline{1}2},N_{\overline{1}\overline{2}}\}}=1N_{\overline{1}2} N_{\overline{1}\overline{2}}N_{\overline{1}2} N_{\overline{1}\overline{2}}1$. Therefore, $1$ does not alternate with $N_{\overline{1}2}$ and $N_{{\overline{1}\overline{2}}}$ in $w$.
        \item In the graph $\overline{B(K_m, K_2)}$, $2\nsim N_{1\overline{2}}$ and $2\nsim N_{{\overline{1}\overline{2}}}$. In the word $w$, $w_{\{2, N_{1\overline{2}},N_{\overline{1}\overline{2}}\}}=2N_{\overline{1}\overline{2}}N_{1\overline{2}}N_{\overline{1}\overline{2}}N_{1\overline{2}}2$. Therefore, $2$ does not alternate with $N_{1\overline{2}}$ and $N_{{\overline{1}\overline{2}}}$ in $w$.
    \end{enumerate}
\end{proof}
\subsection{Clique size $3$}
We are considering $K_n=K_3$, and want to find for different $K_m$, does $\overline{B(K_m, K_3)}$ word-representable. We assume the vertices of $K_3$ are $1$, $2$ and $3$. Also, $N_i$ denotes the set of the vertices that are adjacent to $i$. Now, the vertices of $K_m$ are divided in the following:
\begin{enumerate}
    \item $N_{1\overline{2}\overline{3}}$ is the set of the vertices $v$ such that $v\in V(K_m)$ and $v\in N_1\setminus (N_2\cup N_3)$.
    \item $N_{\overline{1}2\overline{3}}$ is the set of the vertices $v$ such that $v\in V(K_m)$ and $v\in N_2\setminus (N_1\cup N_3)$.
    \item $N_{\overline{1}\overline{2}3}$ is the set of the vertices $v$ such that $v\in V(K_m)$ and $v\in N_3\setminus (N_1\cup N_2)$.
    \item $N_{12\overline{3}}$ is the set of the vertices $v$ such that $v\in V(K_m)$ and $v\in (N_1\cap N_2)\setminus N_3$. 
    \item $N_{\overline{1}23}$ is the set of the vertices $v$ such that $v\in V(K_m)$ and $v\in (N_2\cap N_3)\setminus N_1$. 
    \item $N_{1\overline{2}3}$ is the set of the vertices $v$ such that $v\in V(K_m)$ and $v\in (N_1\cap N_3)\setminus N_2$. 
    \item $N_{123}$ is the set of the vertices $v$ such that $v\in V(K_m)$ and $v\in N_1\cap N_2\cap N_3$.
    \item $N_{\overline{1}\overline{2}\overline{3}}$ is the set of the vertices $v$ such that $v\in V(K_m)$ and $v\notin N_1\cup N_2\cup N_3$.
\end{enumerate}

Consider the non-word-representable co-bipartite graph $\overline{T_1}$ shown in Figure \ref{nonRepTri1}. This graph is partitioned into two cliques, $K_3$ and $K_4$. There exists a vertex $v \in V(K_4)$ that is not connected to any vertices in clique $K_3$. Additionally, the non-word-representable co-bipartite graph $\overline{G_1}$ has a vertex that is adjacent to all others. 

So, we consider the situation where both sets $N_{123}$ and $N_{\overline{1}\overline{2}\overline{3}}$ are empty. In this context, we aim to prove that the bipartite graph $\overline{B(K_m, K_3)}$ is word-representable.
\begin{theorem}
    The co-biparitite graph $\overline{B(K_m, K_3)}$, is word-representable where $N_{123}=\phi$ and $N_{\overline{1}\overline{2}\overline{3}}=\phi$.
\end{theorem}
\begin{proof}
Our claim is that the word
$w= 1 N_{1\overline{2}3} 2 N_{1\overline{2}\overline{3}} N_{12\overline{3}} 3 N_{\overline{1}2\overline{3}} 1 N_{\overline{1}23} 2 N_{\overline{1}\overline{2}3} N_{1\overline{2}3} 3 N_{1\overline{2}\overline{3}} N_{12\overline{3}} 1 N_{\overline{1}2\overline{3}} N_{\overline{1}23} N_{\overline{1}\overline{2}3}	
N_{1\overline{2}3} $ $N_{1\overline{2}\overline{3}} 2 N_{12\overline{3}} N_{\overline{1}2\overline{3}} 3 N_{\overline{1}23} N_{\overline{1}\overline{2}3}$ represents the graph $\overline{B(K_m, K_3)}$, where $N_{1\overline{2}\overline{3}}$, $N_{\overline{1}2\overline{3}}$, $N_{12\overline{3}}$, $N_{\overline{1}\overline{2}3}$, $N_{\overline{1}23}$ and $N_{1\overline{2}3}$ are defined above. Now, in the graph $\overline{B(K_m, K_3)}$, if $x$ and $y$, $x,y \in V(\overline{B(K_m, K_3))}$ are connected by an edge then in the word $w$, $x$ and $y$ should alternate. Every possible case for each edge in $E(\overline{B(K_m, K_3))}$ is considered in the following.
\begin{enumerate}
    \item In the graph $\overline{B(K_m, K_3)}$, $i\sim j$, $1\leq i<j\leq 3$. Now, in the word $w$, $w_{\{1,2,3\}}=(123)^3$. So, $1$, $2$ and $3$ are alternating with each other in $w$. 
    \item In the graph $\overline{B(K_m, K_3)}$, $i\sim j$, $\forall i,j\in V(K_m), i\neq j$. We know that $V(K_m)= N_{1\overline{2}\overline{3}}\cup N_{\overline{1}2\overline{3}}\cup N_{12\overline{3}}\cup N_{\overline{1}\overline{2}3}\cup N_{\overline{1}23}\cup N_{1\overline{2}3}$, as $N_{123}=\phi$. Now, in the word $w$, $w_{\{N_{1\overline{2}\overline{3}}, N_{\overline{1}2\overline{3}}, N_{12\overline{3}}, N_{\overline{1}\overline{2}3}, N_{\overline{1}23}, N_{1\overline{2}3}\}}=(N_{1\overline{2}3}  N_{1\overline{2}\overline{3}} N_{12\overline{3}}  N_{\overline{1}2\overline{3}}  N_{\overline{1}23}  N_{\overline{1}\overline{2}3})^3$. Hence, $i$ and $j$, $\forall i,j\in V(K_m), i\neq j$ are alternating in $w$.
    \item In the graph $\overline{B(K_m, K_3)}$, $1\sim i$, $i\in N_{1\overline{2}\overline{3}}\cup N_{12\overline{3}}\cup N_{1\overline{2}3}$. Now, in the word $w$, $w_{\{1,N_{1\overline{2}\overline{3}}, N_{12\overline{3}}, N_{1\overline{2}3}\}}=(1 N_{1\overline{2}3} N_{1\overline{2}\overline{3}} N_{12\overline{3}})^3$. Therefore, $1$ and $i$ are alternating in $w$.
    \item In the graph $\overline{B(K_m, K_3)}$, $2\sim i$, $i\in N_{\overline{1}2\overline{3}}\cup N_{12\overline{3}}\cup N_{\overline{1}23}$. Now, in the word $w$, $w_{\{2, N_{12\overline{3}},N_{\overline{1}2\overline{3}},N_{\overline{1}23}\}}=(2 N_{12\overline{3}}N_{\overline{1}2\overline{3}}N_{\overline{1}23})^3$. Therefore, $2$ and $i$ are alternating in $w$.
    \item In the graph $\overline{B(K_m, K_3)}$, $3\sim i$, $i\in N_{\overline{1}\overline{2}3}\cup N_{\overline{1}23}\cup N_{1\overline{2}3}$. Now, in the word $w$, $w_{\{3,N_{1\overline{2}3}, N_{\overline{1}23},N_{\overline{1}\overline{2}3}\}}=(N_{1\overline{2}3} 3 N_{\overline{1}23}N_{\overline{1}\overline{2}3})^3$. Therefore, $3$ and $i$ are alternating in $w$.
\end{enumerate}

 Since alternation is preserved in $w$ for the edges of the graph $\overline{B(K_m, K_3)}$, we now need to show that for the non-edges of the graph $\overline{B(K_m, K_3)}$, non-alteration is also preserved in $w$. The possible cases are the following.
 \begin{enumerate}
     \item In the graph $\overline{B(K_m, K_3)}$, $1\nsim i$, $i\in N_{\overline{1}2\overline{3}}\cup N_{\overline{1}\overline{2}3}\cup N_{\overline{1}23}$. Now, in the word $w$, $1$ and $i$ should not alternate. As, $w_{\{1,N_{\overline{1}2\overline{3}}, N_{\overline{1}\overline{2}3}, N_{\overline{1}23}\}}=1N_{\overline{1}2\overline{3}}1N_{\overline{1}23}N_{\overline{1}\overline{2}3}1N_{\overline{1}2\overline{3}}N_{\overline{1}23}N_{\overline{1}\overline{2}3}N_{\overline{1}2\overline{3}}N_{\overline{1}23}N_{\overline{1}\overline{2}3}$, therefore, $1$ is not alternating with $i$ in the word $w$. 
     \item In the graph $\overline{B(K_m, K_3)}$, $2\nsim i$, $i\in N_{1\overline{2}\overline{3}}\cup N_{\overline{1}\overline{2}3}\cup N_{1\overline{2}3}$. Now, in the word $w$, $w_{\{2,N_{1\overline{2}\overline{3}}, N_{\overline{1}\overline{2}3}, N_{1\overline{2}3}\}}=N_{1\overline{2}3}2N_{1\overline{2}\overline{3}}2N_{\overline{1}\overline{2}3}N_{1\overline{2}3}N_{1\overline{2}\overline{3}}N_{\overline{1}\overline{2}3}N_{1\overline{2}3}N_{1\overline{2}\overline{3}}2N_{\overline{1}\overline{2}3}$, therefore, $2$ is not alternating with $i$ in the word $w$. 
     \item In the graph $\overline{B(K_m, K_3)}$, $3\nsim i$, $i\in N_{1\overline{2}\overline{3}}\cup N_{\overline{1}2\overline{3}}\cup N_{12\overline{3}}$. Now, in the word $w$, $w_{3,N_{1\overline{2}\overline{3}}, N_{\overline{1}2\overline{3}}, N_{12\overline{3}}}=N_{1\overline{2}\overline{3}}N_{12\overline{3}}3$ $N_{\overline{1}2\overline{3}}3N_{1\overline{2}\overline{3}}N_{12\overline{3}}N_{\overline{1}2\overline{3}}N_{1\overline{2}\overline{3}}N_{12\overline{3}}N_{\overline{1}2\overline{3}}3$, therefore, $3$ is not alternating with $i$ in the word $w$. 
 \end{enumerate}
\end{proof}

\section{Semi-trasitive orientation of a word-representable co-bipartite graph}
In this section, we studied the semi-transitive orientation of a word-representable co-bipartite graph. First, we present similar work done in the case of split graphs. Word-representable split graphs are studied in the paper \cite{kitaev2021word}. In that paper, it was proved that for the word-representable split graph $S_n=(E_{n-m},K_m)$, vertices of $E_{n-m}$ can be classified as described in Lemma \ref{semi-tran-groups}.
Additionally, according to Lemma~\ref{lem-tran-orie}, any semi-transitive orientation of the graph $S_n = (E_{n-m}, K_m)$ induces a transitive orientation on $K_m$. This can be represented schematically in Figure~\ref{schem-structure}, where the longest directed path in $ K_m $ is represented as $ \vec{P} $. It is important to note that although the other edges in $K_m$ are not shown, they do exist.

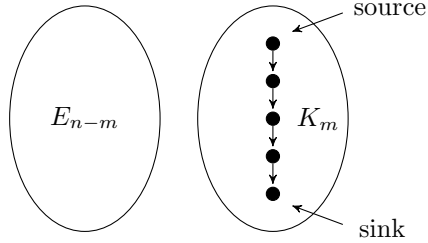
\begin{figure}[H]
\begin{center}

\begin{tikzpicture}[->,>=stealth', shorten >=1pt,node distance=0.5cm,auto,main node/.style={fill,circle,draw,inner sep=0pt,minimum size=5pt}]

\node (1) {};
\node[main node] (2) [right of=1,xshift=5mm] {};
\node[main node] (3) [above of=2] {};
\node[main node] (4) [above of=3] {};
\node[main node] (5) [below of=2] {};
\node[main node] (6) [below of=5] {};

\node (7) [above of=4,yshift=-4mm,xshift=1mm] {};
\node (8) [above right of=7,xshift=6mm]{};
\node (9) [right of=8]{source};

\node (10) [below of=6,yshift=4mm,xshift=1mm] {};
\node (11) [below right of=10,xshift=6mm]{};
\node (12) [right of=11,xshift=-1mm]{sink};

\node [right of=2,xshift=1mm]{$K_m$};
\node [right of=2,xshift=-30mm]{$E_{n-m}$};

\path
(8) edge (7)
(11) edge (10);

\draw [] (-1.5,0) ellipse (10mm and 15mm); 
\draw [] (1,0) ellipse (10mm and 15mm); 

\path
(4) edge (3)
(3) edge (2)
(2) edge (5)
(5) edge (6);

\end{tikzpicture}

\caption{\label{schem-structure} A schematic structure of a semi-transitively oriented split graph}
\end{center}
\end{figure}
Lemmas~\ref{semi-tran-groups} and~\ref{relative-order} below describe the structure of semi-transitive orientations in arbitrary word-representable split graph.

\begin{lemma}\label{semi-tran-groups}[\cite{kitaev2021word}, Lemma 7.1 ] Any semi-transitive orientation of $S_n=(E_{n-m},K_m)$ subdivides the set of all vertices in $E_{n-m}$ into three, possibly empty, groups of the types shown schematically in Figure~\ref{3-groups1}. In that figure,  
\begin{itemize}
\item similarly to Figure~\ref{schem-structure}, the vertical oriented paths are a schematic way to show (parts of) $\vec{P}$;
\item the vertical oriented paths in the types A and B represent up to $m$ consecutive vertices in $\vec{P}$; 
\item the vertical oriented path in the type C contains all $m$ vertices in $\vec{P}$, and it is subdivided into three groups of consecutive vertices, with the middle group possibly containing no vertices; the group of vertices containing the source (resp., sink) is the {\em source-group} (resp., {\em sink-group});
\item The vertices to the left of the vertical paths are from  $E_{n-m}$. 
\end{itemize} 

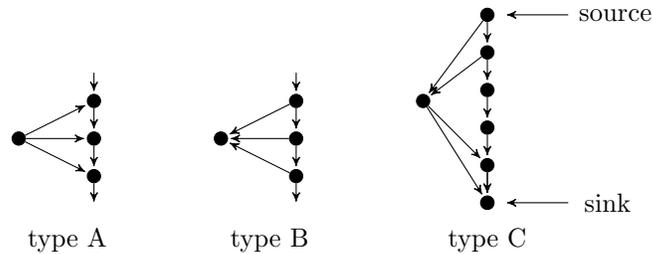
\begin{figure}[H]
\begin{center}

\begin{tabular}{ccc}

\begin{tikzpicture}[->,>=stealth', shorten >=1pt,node distance=0.5cm,auto,main node/.style={fill,circle,draw,inner sep=0pt,minimum size=5pt}]

\node[main node] (1) {};
\node[main node] (2) [right of=1,xshift=5mm] {};
\node[main node] (3) [above of=2] {};
\node (4) [above of=3] {};
\node[main node] (5) [below of=2] {};
\node (6) [below of=5] {};

\node (7) [above of=4,yshift=-4mm,xshift=1mm] {};
\node (8) [above right of=7,xshift=6mm]{};

\node (10) [below of=6,yshift=4mm,xshift=1mm] {};
\node (11) [below right of=10,xshift=6mm]{};

\node (11) [below left of=5,yshift=-5mm]{type A};

\path
(4) edge (3)
(3) edge (2)
(2) edge (5)
(5) edge (6);

\path
(1) edge (2)
(1) edge (3)
(1) edge (5);

\end{tikzpicture}

&
\begin{tikzpicture}[->,>=stealth', shorten >=1pt,node distance=0.5cm,auto,main node/.style={fill,circle,draw,inner sep=0pt,minimum size=5pt}]

\node[main node] (1) {};
\node[main node] (2) [right of=1,xshift=5mm] {};
\node[main node] (3) [above of=2] {};
\node (4) [above of=3] {};
\node[main node] (5) [below of=2] {};
\node (6) [below of=5] {};

\node (7) [above of=4,yshift=-4mm,xshift=1mm] {};
\node (8) [above right of=7,xshift=6mm]{};

\node (10) [below of=6,yshift=4mm,xshift=1mm] {};
\node (11) [below right of=10,xshift=6mm]{};

\node (11) [below left of=5,yshift=-5mm]{type B};

\path
(4) edge (3)
(3) edge (2)
(2) edge (5)
(5) edge (6);

\path
(2) edge (1)
(3) edge (1)
(5) edge (1);

\end{tikzpicture}

&

\begin{tikzpicture}[->,>=stealth', shorten >=1pt,node distance=0.5cm,auto,main node/.style={fill,circle,draw,inner sep=0pt,minimum size=5pt}]

\node[main node] (1) {};
\node[main node] (2) [below right of=1,xshift=5mm] {};
\node[main node] (3) [above of=2] {};
\node[main node] (4) [above of=3] {};
\node[main node] (5) [below of=2] {};
\node[main node] (6) [below of=5] {};
\node[main node] (13) [above of=4] {};
\node[main node] (14) [below of=5] {};

\node (7) [above of=13,yshift=-5mm,xshift=1mm] {};
\node (8) [right of=7,xshift=6mm]{};
\node (9) [right of=8]{source};

\node (10) [below of=14,yshift=5mm,xshift=1mm] {};
\node (11) [right of=10,xshift=6mm]{};
\node (12) [right of=11,xshift=-1mm]{sink};

\node [below of=14]{type C};

\path
(8) edge (7)
(11) edge (10);

\path
(13) edge (1)
(4) edge (1)
(1) edge (14)
(1) edge (5)
(4) edge (3)
(3) edge (2)
(2) edge (5)
(5) edge (6)
(13) edge (4)
(5) edge (14);
\end{tikzpicture}
\end{tabular}

\caption{\label{3-groups1} Three types of vertices in $E_{n-m}$ under a semi-transitive orientation of $(E_{n-m},K_m)$. The vertically oriented paths are a schematic way to show (parts of) $\vec{P}$}
\end{center}
\end{figure}
\end{lemma}

\begin{lemma}\label{relative-order}[\cite{kitaev2021word}, Lemma 7.2] Let $S_n=(E_{n-m},K_m)$ be oriented semi-transitively. For a vertex $x\in E_{n-m}$ of the type C, presented schematically as
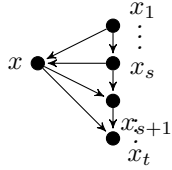
\vspace{-0.8cm}
\begin{center}
\begin{tikzpicture}[->,>=stealth', shorten >=1pt,node distance=0.5cm,auto,main node/.style={fill,circle,draw,inner sep=0pt,minimum size=5pt}]

\node[main node] (1) {};
\node [above of=1,yshift=-5mm,xshift=-3mm] {$x$};
\node[main node] (2) [right of=1,xshift=5mm] {};
\node[main node] (3) [above of=2] {};
\node [right of=2,xshift=-1mm,yshift=-1mm] {$x_s$};
\node [right of=3,xshift=-1mm,yshift=2mm] {$x_1$};
\node [above right of=2,yshift=1mm] {$\vdots$};
\node [right of=6,yshift=5mm,xshift=1mm] {$x_{s+1}$};
\node (4) [above of=3] {};
\node[main node] (5) [below of=2] {};
\node[main node] (6) [below of=5] {};
\node [right of=6,yshift=2mm,xshift=-2mm] {$\vdots$};
\node [below right of=6,yshift=1mm] {$x_t$};

\node (7) [above of=4,yshift=-4mm,xshift=1mm] {};
\node (8) [above right of=7,xshift=6mm]{};

\node (10) [below of=6,yshift=4mm,xshift=1mm] {};
\node (11) [below right of=10,xshift=6mm]{};

\path
(3) edge (2)
(2) edge (5)
(5) edge (6);

\path
(2) edge (1)
(3) edge (1)
(1) edge (5)
(1) edge (6);

\end{tikzpicture}
\end{center}

\noindent 
there is no vertex $y\in E_{n-m}$ of the type A or B, which is connected to both $x_s$ and $x_{s+1}$. Also, there is no vertex $y\in E_{n-m}$ of the type C such that either the source-group or the sink-group of vertices given by $y$ (see the statement of Theorem~\ref{semi-tran-groups} for the definitions) contains both $x_s$ and $x_{s+1}$.
\end{lemma}

\begin{theorem}\label{main-orientation}[\cite{kitaev2021word}, Theorem 7.1] An orientation of a split graph $S_n=(E_{n-m},K_m)$ is semi-transitive if and only if 
\begin{itemize} 
\item $K_m$ is oriented transitively,
\item each vertex in $E_{n-m}$ is of one of the three types presented in Figure~\ref{3-groups1}, and 
\item the restrictions in Lemma~\ref{relative-order} are satisfied. 
\end{itemize}\end{theorem}

\begin{theorem}\label{intercahnging} [\cite{kitaev2021word}, Theorem 7.2]
Let $S_n=(E_{n-m},K_m)$ be semi-transitively oriented. Then, any vertex in $E_{n-m}$ of the type A can be replaced by a vertex of the type B, and vice versa, keeping orientation semi-transitive. 
\end{theorem}

We now characterize the vertices of the word-representable co-bipartite graph $\overline{B(K_m,K_n)}$. Since both partitions are complete graphs, according to Lemma \ref{lem-tran-orie}, any semi-transitive orientation of the graph $\overline{B(K_m,K_n)}$ induces a transitive orientation on both $K_m$ and $K_n$. We use a similar approach to that described in Theorem \ref{semi-tran-groups}. 

We subdivide the vertices of the graph $\overline{B(K_m,K_n)}$ into three types: type $A$, type $B$, and type $C$, as mentioned in Theorem \ref{semi-tran-groups}. For all vertices $x \in V(K_m)$, we identify $\vec{P_1}$ as the longest directed path in $K_n$. Similarly, for all vertices $y \in V(K_n)$, we define $\vec{P_2}$ as the longest directed path in $K_m$.

Our observation is that Theorem \ref{semi-tran-groups} also holds for the graph $\overline{B(K_m,K_n)}$ when considering the vertices of type $A$, type $B$, and type $C$, as previously described.

\begin{observation}\label{obs41}
    Any semi-transitive orientation of $\overline{B(K_m,K_n)}$ subdivides the set of all vertices into three, possibly empty, groups of the types shown schematically in Figure~\ref{3-groups}. 
\begin{itemize}
\item similarly to Figure~\ref{schem-structure}, the vertical oriented paths are a schematic way to show (parts of) $\vec{P_1}$ for $V(K_m)$ and $\vec{P_2}$ for $V(K_n)$ ;
\item the vertical oriented paths in the types $A$ and $B$ represent up to $l$ consecutive vertices in $\vec{P_1}$ for $V(K_m)$ and $\vec{P_2}$ for $V(K_n)$; 
\item the vertical oriented path in the type $C$ contains all $l$ vertices in $\vec{P_1}$ for $V(K_m)$ and $\vec{P_2}$ for $V(K_n)$, and it is subdivided into three groups of consecutive vertices, with the middle group possibly containing no vertices; the group of vertices containing the source (resp., sink) is the {\em source-group} (resp., {\em sink-group});
\end{itemize} 

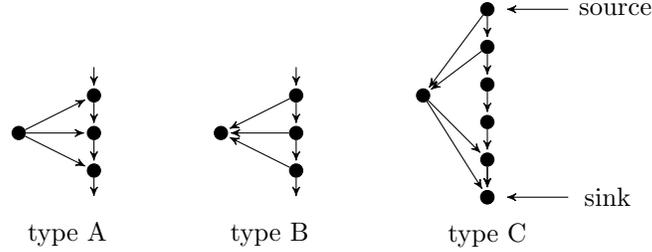
\begin{figure}[H]
\begin{center}

\begin{tabular}{ccc}

\begin{tikzpicture}[->,>=stealth', shorten >=1pt,node distance=0.5cm,auto,main node/.style={fill,circle,draw,inner sep=0pt,minimum size=5pt}]

\node[main node] (1) {};
\node[main node] (2) [right of=1,xshift=5mm] {};
\node[main node] (3) [above of=2] {};
\node (4) [above of=3] {};
\node[main node] (5) [below of=2] {};
\node (6) [below of=5] {};

\node (7) [above of=4,yshift=-4mm,xshift=1mm] {};
\node (8) [above right of=7,xshift=6mm]{};

\node (10) [below of=6,yshift=4mm,xshift=1mm] {};
\node (11) [below right of=10,xshift=6mm]{};

\node (11) [below left of=5,yshift=-5mm]{type A};

\path
(4) edge (3)
(3) edge (2)
(2) edge (5)
(5) edge (6);

\path
(1) edge (2)
(1) edge (3)
(1) edge (5);

\end{tikzpicture}

&
\begin{tikzpicture}[->,>=stealth', shorten >=1pt,node distance=0.5cm,auto,main node/.style={fill,circle,draw,inner sep=0pt,minimum size=5pt}]

\node[main node] (1) {};
\node[main node] (2) [right of=1,xshift=5mm] {};
\node[main node] (3) [above of=2] {};
\node (4) [above of=3] {};
\node[main node] (5) [below of=2] {};
\node (6) [below of=5] {};

\node (7) [above of=4,yshift=-4mm,xshift=1mm] {};
\node (8) [above right of=7,xshift=6mm]{};

\node (10) [below of=6,yshift=4mm,xshift=1mm] {};
\node (11) [below right of=10,xshift=6mm]{};

\node (11) [below left of=5,yshift=-5mm]{type B};

\path
(4) edge (3)
(3) edge (2)
(2) edge (5)
(5) edge (6);

\path
(2) edge (1)
(3) edge (1)
(5) edge (1);

\end{tikzpicture}

&

\begin{tikzpicture}[->,>=stealth', shorten >=1pt,node distance=0.5cm,auto,main node/.style={fill,circle,draw,inner sep=0pt,minimum size=5pt}]

\node[main node] (1) {};
\node[main node] (2) [below right of=1,xshift=5mm] {};
\node[main node] (3) [above of=2] {};
\node[main node] (4) [above of=3] {};
\node[main node] (5) [below of=2] {};
\node[main node] (6) [below of=5] {};
\node[main node] (13) [above of=4] {};
\node[main node] (14) [below of=5] {};

\node (7) [above of=13,yshift=-5mm,xshift=1mm] {};
\node (8) [right of=7,xshift=6mm]{};
\node (9) [right of=8]{source};

\node (10) [below of=14,yshift=5mm,xshift=1mm] {};
\node (11) [right of=10,xshift=6mm]{};
\node (12) [right of=11,xshift=-1mm]{sink};

\node [below of=14]{type C};

\path
(8) edge (7)
(11) edge (10);

\path
(13) edge (1)
(4) edge (1)
(1) edge (14)
(1) edge (5)
(4) edge (3)
(3) edge (2)
(2) edge (5)
(5) edge (6)
(13) edge (4)
(5) edge (14);
\end{tikzpicture}
\end{tabular}

\caption{\label{3-groups} Three types of vertices in $K_m$ and $K_n$ }
\end{center}
\end{figure}

\end{observation}

\begin{proof}
    The same proof of Theorem \ref{semi-tran-groups} is worked for this observation. 
\end{proof}
One important note is that if $ x $ is a type $ C $ vertex in a semi-transitive orientation of the graph $ \overline{B(K_m, K_n)} $, and $ P = a_1a_2\cdots a_l $ represents a directed path for $ x $, then $ a_1 $ and $ a_l $ must be a source and a sink, respectively, in the transitive orientation of the complete graph $ K_i $, where $ \{a_1, a_l\} \in V(K_i) $.

\begin{lemma}\label{lm41}
    Suppose the graph $\overline{B(K_m,K_n)}$ is word-representable and $S$ is a semi-transitive orientation of the graph $\overline{B(K_m,K_n)}$. Suppose $\{x,y\}\in V(K_m)$ and $x$ is a type $A$ vertex and $y$ is a type $B$ vertex and the orientation of the edge between $x$ and $y$ is $y\rightarrow x$. Then, $x$ and $y$ do not have any common neighbour among the vertices of $K_n$.
\end{lemma}
\begin{proof}
    Suppose, $x$ and $y$ have a common neighbour $x_s\in V(K_n)$. As, $x$ is type $A$, the edge between $x$ and $x_s$ is oriented as $x\rightarrow x_s$. Now, the edge between $y$ and $x_s$ are oriented as $x_s\rightarrow y$ because $y$ is type $B$. Then, the induced subgraph $y$, $x$ and $x_s$ has a cyclic orientation. However, it is not possible in the semi-transitive orientation $S$.
\end{proof}

\begin{lemma} \label{lm42}
    Suppose the graph $\overline{B(K_m,K_n)}$ is word-representable and $S$ is a semi-transitive orientation of the graph $\overline{B(K_m,K_n)}$. Suppose $\{x,y\}\in V(K_m)$ and $\{x_s,x_{s+1}\}\in V(K_n)$, and the orientation of the edges between $\{x,y\}$ is $x \rightarrow y$, and between $\{x_s,x_{s+1}\}$ is $x_s \rightarrow x_{s+1}$ in the semi-transitive orientation $S$. Then, the semi-transitive orientation $S$ follows the following conditions:
    \begin{enumerate}
        \item If the orientation of the edges between $\{x_s,x\}$ is $x_s\rightarrow x$ and between $\{y,x_{s+1}\}$ is $y\rightarrow x_{s+1}$ in the semi-transitive orientation $S$, then $x\rightarrow x_{s+1}$ and $x_s\rightarrow y$ in the semi-transitive orientation $S$.
        \item If the orientation of the edges between $\{x,x_{s+1}\}$ is $x\rightarrow x_{s+1}$ and between $\{x_s,x\}$ is $y\rightarrow x_s$ in the semi-transitive orientation $S$, then $x\rightarrow x_s$ and $y\rightarrow x_{s+1}$ in the semi-transitive orientation $S$.
    \end{enumerate}
\end{lemma}
\begin{proof}
    Since $S$ is the semi-transitive orientation of the graph $\overline{B(K_m,K_n)}$, there does not exist any shortcut or cyclic orientation in $S$. Now, we are proving these two conditions in the following:
   \begin{enumerate}
       \item According to the mentioned orientation in the first condition, we can obtain the following orientation among the vertices $x$, $y$, $x_s$ and $x_{s+1}$.

 \begin{figure}[H]
\begin{center}
\begin{tikzpicture}[node distance=1cm,auto,main node/.style={circle,draw,inner sep=1pt,minimum size=5pt}]

\node (1) {$x_s$};
\node (2) [ right of=1] {$x$};
\node (3) [ right of=2] {$y$};
\node (4) [ right of=3] {$x_{s+1}$};

\draw[->] (1) -- (2);
\draw[->] (2) -- (3);
\draw[->] (3) -- (4);
\draw [->] (1) to [out=60, in=120] (4);

\end{tikzpicture}

\caption{\label{fig3} orientation among the vertices $x$, $y$, $x_s$ and $x_{s+1}$}
\end{center}
\end{figure}
We can clearly see that this orientation creates a shortcut if $x$ and $y$ are not adjacent with $x_{s+1}$ and $x_s$, respectively. Now, if we orient the edge between $\{x_{s+1},x\}$ as $x_{s+1}\rightarrow x$, then $x$, $y$ and $x_{s+1}$ create a cyclic orientation. However, it is not possible in the semi-transitive orientation $S$. Therefore,  the edge between $\{x_{s+1},x\}$ is oriented as $x\rightarrow x_{s+1}$. Similarly, if we orient the edge between $\{y,x_s\}$ as $y\rightarrow x_s$, then $x_s$, $x$ and $y$ create a cyclic orientation. However, it is not possible in the semi-transitive orientation $S$. Therefore, the edge between $\{y,x_s\}$ is oriented as $x_s\rightarrow y$.

\item According to the mentioned orientation in the second condition, we can obtain the following orientation among the vertices $x$, $y$, $x_s$ and $x_{s+1}$.

\begin{figure}[H]
\begin{center}
\begin{tikzpicture}[node distance=1cm,auto,main node/.style={circle,draw,inner sep=1pt,minimum size=5pt}]

\node (1) {$x$};
\node (2) [ right of=1] {$y$};
\node (3) [ right of=2] {$x_s$};
\node (4) [ right of=3] {$x_{s+1}$};

\draw[->] (1) -- (2);
\draw[->] (2) -- (3);
\draw[->] (3) -- (4);
\draw [->] (1) to [out=60, in=120] (4);

\end{tikzpicture}

\caption{\label{fig4} orientation among the vertices $x$, $y$, $x_s$ and $x_{s+1}$}
\end{center}
\end{figure}
We can clearly see that this orientation creates a shortcut if $x$ and $y$ are not adjacent with $x_s$ and $x_{s+1}$, respectively. Now, if we orient the edge between $\{x,x_s\}$ as $x_s\rightarrow x$, then $x$, $y$, and $x_s$ create a cyclic orientation. However, it is not possible in the semi-transitive orientation $S$. Therefore, the edge between $\{x,x_s\}$ is oriented as $x\rightarrow x_s$. Similarly, if we orient the edge between $\{x_{s+1},y\}$ as $x_{s+1}\rightarrow y$, then $y$, $x_s$, and $x_{s+1}$ create a cyclic orientation. However, it is not possible in the semi-transitive orientation $S$. Therefore, the edge between $\{x_{s+1},y\}$ is oriented as $y\rightarrow x_{s+1}$.
   \end{enumerate}
\end{proof}

\begin{lemma}\label{lm43}
Suppose the graph $\overline{B(K_m,K_n)}$ is word-representable and $S$ is a semi-transitive orientation of the graph. Let $x\in V(\overline{B(K_m,K_n)})$, and $x$ is a type $C$ vertex. Suppose $x_s$ and $x_{s+1}$ are the last and first vertex of the source and sink group, respectively.
\begin{figure}[H]
\begin{center}
\begin{tikzpicture}[node distance=1cm,auto,main node/.style={fill,circle,draw,inner sep=1pt,minimum size=5pt}]

\node [main node](1) {};
\node [main node](2) [ below of=1] {};
\node [right of=2,xshift=-5mm] {$x_s$};
\node [main node](3) [ below of=2] {};
\node [main node](4) [ below of=3] {};
\node [right of=4,xshift=-5mm] {$x_{s+1}$};
\node [main node](5) [ below of=4] {};
\node [main node](6) [ left of=3] {};
\node [left of=6,xshift=5mm] {$x$};

\draw[->] (1) -- (2);
\draw[->] (2) -- (3);
\draw[->] (3) -- (4);
\draw[->] (4) -- (5);
\draw[->] (1) -- (6);
\draw[->] (2) -- (6);
\draw[->] (6) -- (4);
\draw[->] (6) -- (5);

\end{tikzpicture}

\caption{\label{fig5} type $C$ vertex $x$.}
\end{center}
\end{figure}
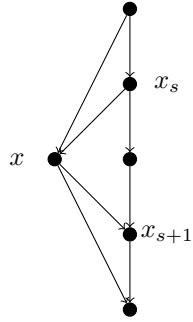
Now, for $y\in V(\overline{B(K_m,K_n)})$, if $x$ and $y$ are in the same partition in the graph $V(\overline{B(K_m,K_n)})$, then the following conditions hold for $y$.

\begin{enumerate}
    \item If the edges between $x$ and $y$ are oriented as $x\rightarrow y$, $y$ is not a type $A$ vertex that is adjacent to both $x_s$ and $x_{s+1}$. Also, if $y$ is type $C$, then $x_s$ cannot contain in the sink group of $y$.
    \item If the edges between $x$ and $y$ are oriented as $y\rightarrow x$, $y$ is not a type $B$ vertex that is adjacent to both $x_s$ and $x_{s+1}$. Also, if $y$ is type $C$, then $x_{s+1}$ cannot contain in the source group of $y$.
\end{enumerate}
\end{lemma}
 \begin{proof}
Since $S$ is the semi-transitive orientation of the graph $\overline{B(K_m,K_n)}$, there does not exist any shortcut or cyclic orientation in $S$. Now, we are proving these two conditions in the following:
\begin{enumerate}
    \item Suppose the vertex $y$ is type $A$ and $y$ is adjacent to $x_s$. Then, the orientation of the edge between $y$ and $x_s$ is $y\rightarrow x_s$.  As $x_s\rightarrow x$, the subgraph induced by the vertices $x_s$, $x$, and $y$ has a cyclic orientation. However, it is not possible in the semi-transitive orientation $S$. Suppose the vertex $y$ is type $A$ and $y$ is adjacent to $x_{s+1}$.  Then, the orientation of the edge between $y$ and $x_{s+1}$ is $y\rightarrow x_{s+1}$. As $x_s\rightarrow x$ and $x_s\rightarrow x_{s+1}$, the subgraph induced by the vertices $x_s$, $x$, $y$, and $x_{s+1}$ create a short cut. Because $y$ and $x_s$ are not adjacent. However, it is not possible in the semi-transitive orientation $S$.
    Now, suppose the vertex $y$ is type $C$ and $x_s$ is in the sink group of $y$. Then, the orientation of the edge between $y$ and $x_s$ is $y\rightarrow x_s$. But, we proved earlier that this case cannot occur in the semi-transitive orientation $S$.

    \item Suppose the vertex $y$ is type $B$ and $y$ is adjacent to $x_{s+1}$. Then, the orientation of the edge between $y$ and $x_{s+1}$ is $x_{s+1}\rightarrow y$.  As $x\rightarrow x_{s+1}$, the subgraph induced by the vertices $x_{s+1}$, $y$, and $x$ has a cyclic orientation. However, it is not possible in the semi-transitive orientation $S$. Suppose the vertex $y$ is type $B$ and $y$ is adjacent to $x_s$.  Then, the orientation of the edge between $y$ and $x_s$ is $x_s\rightarrow y$. As $x\rightarrow x_s$ and $x_s\rightarrow x_{s+1}$, the subgraph induced by the vertices $x_s$, $y$, $x$, and $x_{s+1}$ create a short cut. Because $y$ and $x_{s+1}$ are not adjacent. However, it is not possible in the semi-transitive orientation $S$.
    Now, suppose the vertex $y$ is type $C$ and $x_{s+1}$ is in the source group of $y$. Then, the orientation of the edge between $y$ and $x_{s+1}$ is $x_{s+1}\rightarrow y$. But we proved earlier that this case cannot occur in the semi-transitive orientation $S$.
\end{enumerate}
\end{proof}

\begin{theorem}\label{thm45}
    A co-bipartite graph $\overline{B(K_m,K_n)}$ has a semi-transitive orientation if and only if 
\begin{itemize} 
\item $K_m$ and $K_n$ are oriented transitively,
\item Each vertex in $\overline{B(K_m,K_n)}$ is of one of the three types described in Observation \ref{obs41}, and 
\item The conditions present in Lemmas \ref{lm41}, \ref{lm42} and \ref{lm43} are satisfied. 
\end{itemize}
\end{theorem}

\begin{proof}
    From the Observation \ref{obs41} and Lemmas \ref{lm41}, \ref{lm42} and \ref{lm43}, we can show that a semi-transitive orientation of the graph $\overline{B(K_m,K_n)}$ satisfies all the conditions.\\
    Now we need to prove that if these conditions are stratified, then that orientation is semi-transitive. Suppose $S$ is an orientation of the graph $\overline{B(K_m,K_n)}$ that satisfies the mentioned condition. We assume that $S$ is not a semi-transitive orientation. So, there exists a shortcut in the orientation $S$. Suppose the induced subgraph of $x_1$,$x_2$, $\ldots$, $x_p$ vertices of the graph  $\overline{B(K_m,K_n)}$ contain the following shortcut $P$.

\begin{figure}[H]
\begin{center}
\begin{tikzpicture}[node distance=1cm,auto,main node/.style={circle,draw,inner sep=1pt,minimum size=5pt}]

\node (1) {$x_1$};
\node (2) [ right of=1] {$x_2$};
\node (3) [ right of=2] {$\cdots$};
\node (4) [ right of=3] {$x_i$};
\node (5) [ right of=4] {$\cdots$};
\node (6) [ right of=5] {$x_p$};

\draw[->] (1) -- (2);
\draw[->] (2) -- (3);
\draw[->] (3) -- (4);
\draw[->] (4) -- (5);
\draw[->] (5) -- (6);
\draw [->] (1) to [out=60, in=120] (6);

\end{tikzpicture}

\caption{\label{fig6} orientation among the vertices $x$, $y$, $x_s$ and $x_{s+1}$}
\end{center}
\end{figure}
In the shortcut $P$, vertices should be present from both $K_m$ and $K_n$. Otherwise, $P$ is a complete graph, not a shortcut. First, we check for each $x_i$, $1\leq i\leq p$ does there exist a $x_j$, $1\leq i<j\leq p$ such that $x_i\nsim x_j$. There should be at least one such $x_i$ that needs to be present in the shortcut $P$. We assume that $x_i$ and $x_j$ are the first two vertices we obtained in $P$ such that $x_i\nsim x_j$. Now, we consider the following two cases:\\

\textbf{Case 1:} Suppose $x_1= x_i$. Now, $x_1$ is is adjacent to every $x_a$, $1\leq a<j$. So, we obtain the following shortcut in the shortcut $P$ in the induced subgraph of the vertices $x_1$, $x_{j-1}$, $x_{j}$, $\ldots$ , $x_p$.
\begin{figure}[H]
\begin{center}
\begin{tikzpicture}[node distance=1cm,auto,main node/.style={circle,draw,inner sep=1pt,minimum size=5pt}]

\node (1) {$x_1$};
\node (2) [ right of=1] {$x_{j-1}$};
\node (3) [ right of=2] {$x_{j}$};
\node (4) [ right of=3] {$\cdots$};
\node (5) [ right of=4] {$x_j$};

\draw[->] (1) -- (2);
\draw[->] (2) -- (3);
\draw[->] (3) -- (4);
\draw[->] (4) -- (5);
\draw [->] (1) to [out=60, in=120] (5);

\end{tikzpicture}

\caption{\label{fig7} orientation among the vertices $x$, $y$, $x_s$ and $x_{s+1}$}
\end{center}
\end{figure}

Now, we know that $x_1$ and $x_j$ are in different partitions in the graph $\overline{B(K_m,K_n)}$. Without loss of generality, we assume $x_1\in V(K_m)$ and $x_j\in V(K_n)$. So, the possible cases for $x_{j-1}$, $x_p$ are $\{x_{j-1}, x_p\}\subseteq V(K_m)$ or $\{x_{j-1}, x_p\}\subseteq V(K_n)$ or $x_{j-1}\in V(K_m), x_p\in V(K_n)$ or $x_{j-1}\in V(K_n), x_p\in V(K_m)$. In the following, we prove this shortcut cannot exist for each case.\\
\textbf{Case 1.1:} If $\{x_{j-1}, x_{p}\}\subseteq V(K_m)$, then there exists some $x_q\in V(K_n)$ and $x_r\in V(K_m)$, $j\leq q<r \leq p$, such that the orientation of the edge between $x_q$ and $x_r$ is $x_q\rightarrow x_r$, otherwise, $x_p\notin V(K_m)$. The following orientation occurs for the induced subgraph of the vertices $x_1$, $x_{j-1}$, $x_j$, $x_q$, $x_r$ and $x_p$.

\begin{figure}[H]
\begin{center}
\begin{tikzpicture}[node distance=1cm,auto,main node/.style={fill,circle,draw,inner sep=1pt,minimum size=5pt}]

\node [main node](1) {};
\node [right of=1,xshift=-5mm] {$x_1$};
\node [main node](2) [ below of=1] {};
\node [right of=2,xshift=-5mm] {$x_{j-1}$};
\node [main node](3) [ below of=2] {};
\node [right of=3,xshift=-5mm] {$x_{r}$};
\node [main node](4) [ below of=3] {};
\node [right of=4,xshift=-5mm] {$x_{p}$};
\node [main node](5) [ left of=2] {};
\node [left of=5,xshift=5mm] {$x_j$};
\node [main node](6) [ left of=3] {};
\node [left of=6,xshift=5mm] {$x_q$};

\draw[->] (1) -- (2);
\draw[->] (2) -- (3);
\draw[->] (3) -- (4);
\draw[->] (5) -- (6);
\draw[->] (2) -- (5);
\draw[->] (6) -- (3);

\end{tikzpicture}
\caption{\label{fig8} Induced subgraph of the vertices $x_1$, $x_{j-1}$, $x_j$, $x_q$, $x_r$ and $x_p$.}
\end{center}
\end{figure}
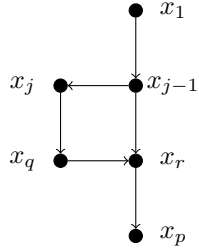
But, according to Lemma \ref{lm42}, $x_j$ and $x_r$ should be adjacent and the edge is oriented as $x_j\rightarrow x_r$. However, $x_j$ becomes type $C$ vertex. So, according to Observation \ref{obs41}, $x_1$ and $x_j$ should be adjacent. it contradicts our assumption regarding the shortcut. \\
\textbf{Case 1.2:} If $\{x_{j-1}, x_{p}\}\subseteq V(K_n)$, then the following orientation occurs for the induced subgraph of the vertices $x_1$, $x_{j-1}$, $x_{j}$ and $x_p$.

\begin{figure}[H]
\begin{center}
\begin{tikzpicture}[node distance=1cm,auto,main node/.style={fill,circle,draw,inner sep=1pt,minimum size=5pt}]

\node [main node](1) {};
\node [right of=1,xshift=-5mm] {$x_{j-1}$};
\node [main node](2) [ below of=1] {};
\node [right of=2,xshift=-5mm] {$x_{j}$};
\node [main node](3) [ below of=2] {};
\node [right of=3,xshift=-5mm] {$x_{p}$};
\node [main node](4) [ left of=2] {};
\node [left of=4,xshift=5mm] {$x_1$};

\draw[->] (1) -- (2);
\draw[->] (2) -- (3);
\draw[->] (4) -- (1);
\draw[->] (4) -- (3);

\end{tikzpicture}
\caption{\label{fig9} induced subgraph of the vertices $x_1$, $x_{j-1}$, $x_{j}$ and $x_p$.}
\end{center}
\end{figure}
But, according to Observation \ref{obs41}, $x_1$ and $x_j$ should be adjacent, as $x_1$ is a type $A$ vertex or $x_1$ is a type $C$ vertex, then $x_j$ should be in the source group of $x_1$. So, it contradicts our assumption regarding the shortcut. \\
\textbf{Case 1.3:} If $x_{j-1}\in V(K_m), x_p\in V(K_n)$, then the following orientation occurs for the induced subgraph of the vertices $x_1$, $x_{j-1}$, $x_j$ and $x_p$.

\begin{figure}[H]
\begin{center}
\begin{tikzpicture}[node distance=1cm,auto,main node/.style={fill,circle,draw,inner sep=1pt,minimum size=5pt}]

\node [main node](1) {};
\node [right of=1,xshift=-5mm] {$x_{j}$};
\node [main node](2) [ below of=1] {};
\node [right of=2,xshift=-5mm] {$x_{p}$};
\node [main node](3) [ left of=1] {};
\node [left of=3,xshift=5mm] {$x_{1}$};
\node [main node](4) [ left of=2] {};
\node [left of=4,xshift=5mm] {$x_{j-1}$};

\draw[->] (1) -- (2);
\draw[->] (3) -- (4);
\draw[->] (3) -- (2);
\draw[->] (4) -- (1);

\end{tikzpicture}
\caption{\label{fig10} induced subgraph of the vertices $x_1$, $x_{j-1}$, $x_{j}$ and $x_p$.}
\end{center}
\end{figure}
But, according to Lemma \ref{lm42}, $x_i$ and $x_j$ should be adjacent. So, it contradicts our assumption regarding the shortcut. \\
\textbf{Case 1.4:} If $x_{j-1}\in V(K_n), x_p\in V(K_m)$, then there exists some $x_q\in V(K_n)$ and $x_r\in V(K_m)$, $j\leq q<r \leq p$, such that the orientation of the edge between $x_q$ and $x_r$ is $x_q\rightarrow x_r$, otherwise, $x_p\notin V(K_m)$. The following orientation occurs for the induced subgraph of the vertices $x_1$, $x_{j-1}$, $x_j$, $x_q$, $x_r$ and $x_p$.

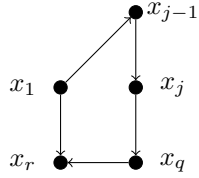
\begin{figure}[H]
\begin{center}
\begin{tikzpicture}[node distance=1cm,auto,main node/.style={fill,circle,draw,inner sep=1pt,minimum size=5pt}]

\node [main node](1) {};
\node [right of=1,xshift=-5mm] {$x_{j-1}$};
\node [main node](2) [ below of=1] {};
\node [right of=2,xshift=-5mm] {$x_{j}$};
\node [main node](3) [ below of=2] {};
\node [right of=3,xshift=-5mm] {$x_{q}$};
\node [main node](4) [ left of=2] {};
\node [left of=4,xshift=5mm] {$x_1$};
\node [main node](5) [ left of=3] {};
\node [left of=5,xshift=5mm] {$x_r$};

\draw[->] (1) -- (2);
\draw[->] (2) -- (3);
\draw[->] (4) -- (5);
\draw[->] (4) -- (1);
\draw[->] (3) -- (5);
\end{tikzpicture}
\caption{\label{fig11} induced subgraph of the vertices $x_1$, $x_{j-1}$, $x_{j}$ and $x_p$.}
\end{center}
\end{figure}
Now, the edge between $x_{j-1}$ and $x_q$ are oriented as $x_{j-1}\rightarrow x_q$. Then, according to Lemma \ref{lm42}, $x_1$ and $x_q$ should be adjacent, and the orientation of the edge between $x_1$ and $x_q$ is $x_1\rightarrow x_q$. But, the induced subgraph of the vertices $x_1$, $x_{j-1}$, $x_j$ and $x_q$ are the same as the subgraph mentioned in Case 1.2. Therefore, $x_1$ and $x_j$ are adjacent. So, it contradicts our assumption regarding the shortcut. \\
\textbf{Case 2:} Suppose $x_1\neq x_i$. Now, every $x_a$, $1\leq a<i$ is adjacent to every $x_b$, $1\leq b\leq p$ in the shortcut $P$. There exists at least one $x_a$ because, in Case 1, we proved that $x_1$ is adjacent to all the vertices present in the shortcut. Also, $x_i$ is is adjacent to every $x_c$, $1\leq c<j$. So, we obtain the following shortcut in the shortcut $P$ among $x_{i-1}$, $x_i$, $x_{j-1}$ and $x_j$ vertices.
\begin{figure}[H]
\begin{center}
\begin{tikzpicture}[node distance=1cm,auto,main node/.style={circle,draw,inner sep=1pt,minimum size=5pt}]

\node (1) {$x_{i-1}$};
\node (2) [ right of=1] {$x_i$};
\node (3) [ right of=2] {$x_{j-1}$};
\node (4) [ right of=3] {$x_j$};

\draw[->] (1) -- (2);
\draw[->] (2) -- (3);
\draw[->] (3) -- (4);
\draw [->] (1) to [out=60, in=120] (4);
\draw [->] (1) to [out=60, in=120] (3);

\end{tikzpicture}

\caption{\label{fig12} orientation among the vertices $x$, $y$, $x_s$ and $x_{s+1}$}
\end{center}
\end{figure}
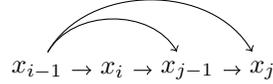

Now, we know that $x_i$ and $x_j$ are in different partition in the graph $\overline{B(K_m,K_n)}$. Without loss of generality, we assume $x_i\in V(K_m)$ and $x_j\in V(K_n)$. So, the possible cases for $x_{i-1}, x_{j-1}$ are $\{x_{i-1}, x_{j-1}\}\subseteq V(K_m)$ or $\{x_{i-1}, x_{j-1}\}\subseteq V(K_n)$ or $x_{i-1}\in V(K_m), x_{j-1}\in V(K_n)$ or $x_{i-1}\in V(K_n), x_{j-1}\in V(K_m)$. In the following we prove for each cases this shortcut cannot exist.\\
\textbf{Case 2.1:} If $\{x_{i-1}, x_{j-1}\}\subseteq V(K_m)$, then the following orientation occurs for the induced subgraph of the vertices $x_{i-1}$, $x_i$, $x_{j-1}$ and $x_j$.

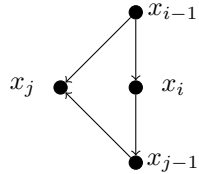
\begin{figure}[H]
\begin{center}
\begin{tikzpicture}[node distance=1cm,auto,main node/.style={fill,circle,draw,inner sep=1pt,minimum size=5pt}]

\node [main node](1) {};
\node [right of=1,xshift=-5mm] {$x_{i-1}$};
\node [main node](2) [ below of=1] {};
\node [right of=2,xshift=-5mm] {$x_i$};
\node [main node](3) [ below of=2] {};
\node [right of=3,xshift=-5mm] {$x_{j-1}$};
\node [main node](4) [ left of=2] {};
\node [left of=4,xshift=5mm] {$x_j$};

\draw[->] (1) -- (2);
\draw[->] (2) -- (3);
\draw[->] (1) -- (4);
\draw[->] (3) -- (4);

\end{tikzpicture}
\caption{\label{fig13} induced subgraph of the vertices $x_{i-1}$, $x_i$, $x_{j-1}$ and $x_j$.}
\end{center}
\end{figure}
But, according to Observation \ref{obs41}, $x_i$ and $x_j$ should be adjacent. So, it contradicts our assumption regarding the shortcut. \\
\textbf{Case 2.2:} If $\{x_{i-1}, x_{j-1}\}\subseteq V(K_n)$, then the following orientation occurs for the induced subgraph of the vertices $x_{i-1}$, $x_i$, $x_{j-1}$ and $x_j$.

\begin{figure}[H]
\begin{center}
\begin{tikzpicture}[node distance=1cm,auto,main node/.style={fill,circle,draw,inner sep=1pt,minimum size=5pt}]

\node [main node](1) {};
\node [right of=1,xshift=-5mm] {$x_{i-1}$};
\node [main node](2) [ below of=1] {};
\node [right of=2,xshift=-5mm] {$x_{j-1}$};
\node [main node](3) [ below of=2] {};
\node [right of=3,xshift=-5mm] {$x_{j}$};
\node [main node](4) [ left of=2] {};
\node [left of=4,xshift=5mm] {$x_i$};

\draw[->] (1) -- (2);
\draw[->] (2) -- (3);
\draw[->] (1) -- (4);
\draw[->] (4) -- (2);

\end{tikzpicture}
\caption{\label{fig14} induced subgraph of the vertices $x_{i-1}$, $x_i$, $x_{j-1}$ and $x_j$.}
\end{center}
\end{figure}
But, according to Observation \ref{obs41}, $x_i$ and $x_j$ should be adjacent. So, it contradicts our assumption regarding the shortcut. \\
\textbf{Case 2.3:} If $x_{i-1}\in V(K_m), x_{j-1}\in V(K_n)$, then the following orientation occurs for the induced subgraph of the vertices $x_{i-1}$, $x_i$, $x_{j-1}$ and $x_j$.

\begin{figure}[H]
\begin{center}
\begin{tikzpicture}[node distance=1cm,auto,main node/.style={fill,circle,draw,inner sep=1pt,minimum size=5pt}]

\node [main node](1) {};
\node [right of=1,xshift=-5mm] {$x_{j-1}$};
\node [main node](2) [ below of=1] {};
\node [right of=2,xshift=-5mm] {$x_{j}$};
\node [main node](3) [left of=1] {};
\node [left of=3,xshift=5mm] {$x_{i-1}$};
\node [main node](4) [ left of=2] {};
\node [left of=4,xshift=5mm] {$x_i$};

\draw[->] (1) -- (2);
\draw[->] (3) -- (4);
\draw[->] (3) -- (1);
\draw[->] (3) -- (2);
\draw[->] (4) -- (1);

\end{tikzpicture}
\caption{\label{fig15} induced subgraph of the vertices $x_{i-1}$, $x_i$, $x_{j-1}$ and $x_j$.}
\end{center}
\end{figure}
But, according to Lemma \ref{lm42}, $x_i$ and $x_j$ should be adjacent. So, it contradicts our assumption regarding the shortcut. \\
\textbf{Case 2.4:} If $x_{i-1}\in V(K_n), x_{j-1}\in V(K_m)$, then the following orientation occurs for the induced subgraph of the vertices $x_{i-1}$, $x_i$, $x_{j-1}$ and $x_j$.

\begin{figure}[H]
\begin{center}
\begin{tikzpicture}[node distance=1cm,auto,main node/.style={fill,circle,draw,inner sep=1pt,minimum size=5pt}]

\node [main node](1) {};
\node [right of=1,xshift=-5mm] {$x_{i-1}$};
\node [main node](2) [ below of=1] {};
\node [right of=2,xshift=-5mm] {$x_{j}$};
\node [main node](3) [left of=1] {};
\node [left of=3,xshift=5mm] {$x_{i}$};
\node [main node](4) [ left of=2] {};
\node [left of=4,xshift=5mm] {$x_{j-1}$};

\draw[->] (1) -- (2);
\draw[->] (3) -- (4);
\draw[->] (1) -- (3);
\draw[->] (4) -- (2);
\draw[->] (1) -- (4);

\end{tikzpicture}
\caption{\label{fig16} induced subgraph of the vertices $x_{i-1}$, $x_i$, $x_{j-1}$ and $x_j$.}
\end{center}
\end{figure}
But, according to Lemma \ref{lm42}, $x_i$ and $x_j$ should be adjacent. So, it contradicts our assumption regarding the shortcut. \\
As, we can see every $x_i$ and $x_j$, $\{x_i,x_j\}\subset V(P)$, are adjacent, therefore, $P$ is not a shortcut. So, the orientation $S$ is semi-transitive.
\end{proof}

	\bibliographystyle{plain}
	\bibliography{ref.bib}
\end{document}